\documentclass{amsart}
\newcommand{\mt}[1]{\mathtt{#1}}

\newcommand{\mbM}{\mathbb{M}}

\usepackage{amsmath}
\usepackage{amsthm}
\usepackage{graphicx}
\usepackage{color}
\usepackage{amsfonts}
\usepackage{amssymb}
\usepackage{mathrsfs}
\usepackage{amscd}
\usepackage{url}

\newcommand{\re}{\mathbb{R}}
\newcommand{\cpx}{\mathbb{C}}

\newcommand{\N}{\mathbb{N}}

\newcommand{\half}{\frac{1}{2}}
\newcommand{\lmd}{\lambda}

\newcommand{\eps}{\epsilon}

\newcommand{\dt}{\delta}
\newcommand{\Dt}{\Delta}
\def\af{\alpha}
\def\bt{\beta}

\def\rank{\mbox{rank}}
\def\vec{\mathbf{vec}}

\newcommand{\sig}{\sigma}
\newcommand{\Sig}{\Sigma}

\newcommand{\reff}[1]{(\ref{#1})}

\newcommand{\mc}[1]{\mathcal{#1}}

\newcommand{\supp}[1]{\mbox{supp}(#1)}

\newcommand{\bdes}{\begin{description}}
\newcommand{\edes}{\end{description}}

\newcommand{\bal}{\begin{align}}
\newcommand{\eal}{\end{align}}

\newcommand{\bnum}{\begin{enumerate}}
\newcommand{\enum}{\end{enumerate}}

\newcommand{\bit}{\begin{itemize}}
\newcommand{\eit}{\end{itemize}}

\newcommand{\bea}{\begin{eqnarray}}
\newcommand{\eea}{\end{eqnarray}}
\newcommand{\be}{\begin{equation}}
\newcommand{\ee}{\end{equation}}

\newcommand{\baray}{\begin{array}}
\newcommand{\earay}{\end{array}}

\newcommand{\bsry}{\begin{subarray}}
\newcommand{\esry}{\end{subarray}}

\newcommand{\bca}{\begin{cases}}
\newcommand{\eca}{\end{cases}}

\newcommand{\bcen}{\begin{center}}
\newcommand{\ecen}{\end{center}}

\newcommand{\bbm}{\begin{bmatrix}}
\newcommand{\ebm}{\end{bmatrix}}

\newcommand{\bmx}{\begin{matrix}}
\newcommand{\emx}{\end{matrix}}

\newcommand{\bpm}{\begin{pmatrix}}
\newcommand{\epm}{\end{pmatrix}}

\newcommand{\btab}{\begin{tabular}}
\newcommand{\etab}{\end{tabular}}

\newtheorem{theorem}{Theorem}[section]

\newtheorem{prop}[theorem]{Proposition}

\theoremstyle{definition}
\newtheorem{example}[theorem]{Example}

\newtheorem{alg}[theorem]{Algorithm}

\begin{document}

\title{Positive Maps and Separable Matrices}

\author{Jiawang Nie}
\address{
Department of Mathematics,  University of California San Diego,  9500
Gilman Drive,  La Jolla,  California 92093,  USA.
} \email{njw@math.ucsd.edu}

\author{Xinzhen Zhang}
\address{
Department of Mathematics, School of Science, Tianjin University, Tianjin 300072, China.
} \email{xzzhang@tju.edu.cn}

\begin{abstract}
A linear map between real symmetric matrix spaces
is {\it positive} if all positive semidefinite matrices are
mapped to positive semidefinite ones.
A real symmetric matrix is {\it separable} if
it can be written as a summation of Kronecker products
of positive semidefinite matrices.
This paper studies how to check if a linear map is positive or not
and  how to check if a matrix is separable or not.
We propose numerical algorithms,
based on Lasserre's type of semidefinite relaxations,
for solving such questions. To check the positivity of a linear map,
we construct a hierarchy of semidefinite relaxations for minimizing
the associated bi-quadratic form over the unit spheres.
We show that the positivity can be detected by solving
a finite number of such semidefinite relaxations.
To check the separability of a matrix,
we construct a hierarchy of semidefinite relaxations.
If it is not separable, we can get a mathematical certificate for that;
if it is, we can get a decomposition for the separability.
\end{abstract}



\maketitle

\section{Introduction}

For an integer $k>0$, denote by $\mc{S}^k$
the space of $k\times k$ real symmetric matrices,
and denote by $\mc{S}_+^k$
the cone of $k\times k$ real symmetric positive semidefinite matrices.
For $X \in \mc{S}^k$,
by $X \succeq 0$ we mean that $X \in \mc{S}_+^k$.

\subsection{Positive maps}
\label{ssc:posmaps}

Let $p, q$ be positive integers. A linear map
\[
\Phi: \, \mc{S}^p \, \to \,  \mc{S}^q
\]
is said to be {\it positive} if
$\Phi(X) \in \mc{S}_+^q$ for all $X \in \mc{S}_+^p$.
An important problem in applications is
checking whether or not a linear map is positive.
It is well-known that checking positivity of linear maps
is equivalent to detecting nonnegativity of bi-quadratic forms
%
% Choi's theorem, which can date back to 1975 in \cite{C75}.
%
(cf.~Choi~\cite{C75}).
As in \cite{C75}, one can show that $\Phi$ is a positive map if and only if
\[
B(x,y) \, := \, y^T \Phi( xx^T ) y \geq 0 \quad
\forall x \in \re^p, \, y \in \re^q.
\]
The above $B(x,y)$ is a bi-quadratic form in two groups of variables
\[
x :=(x_1,\ldots, x_p), \quad y:=(y_1, \ldots, y_q).
\]
Let $E_{ik}$ be the symmetric matrix in $\mc{S}^{p}$
whose $(i,k)$th and $(k,i)$th entries equal to one and
all other entries are zeros. Denote
\be \label{df:Omg}
\Omega:= \{ (i,j,k,l): \,  1\leq i \leq k \leq p,\, 1 \leq j \leq l \leq q \}.
\ee
Then, we can expand $B(x,y)$ as a polynomial in $(x,y)$:
\be \label{nt:yPhi(xx)y}
y^T \Phi( xx^T ) y =
y^T \Big( \sum_{ 1\leq i \leq k \leq p}  x_ix_k \Phi(E_{ik})  \Big) y =
\sum_{ (i,j,k,l) \in \Omega } b_{ijkl} x_i y_j x_ky_l,
\ee
where each $b_{ijkl} = \big( \Phi(E_{ik})  \big)_{jl}
+ \big( \Phi(E_{ik})  \big)_{lj}$.
The coefficients $b_{ijkl}$ are uniquely determined by the linear map $\Phi$,
i.e., $\Phi$ uniquely determines the array
\be \label{B=bijkl:E}
\mc{B} =(b_{ijkl})_{ (i,j,k,l) \in \Omega },
\ee
and vice versa. The array $\mc{B}$ can be thought of as a vector
in the space $\re^{\Omega}$. Denote by $\mathscr{P}^{p,q}$
the set of all positive linear maps from $\mc{S}^p$ to $\mc{S}^q$.
The set $\mathscr{P}^{p,q}$ is a closed convex cone,
which can be implied by Proposition~3.2 of \cite{linmomopt}.
A goal of this paper is to check
the membership in $\mathscr{P}^{p,q}$.
This question is related to bi-quadratic optimization,
which was studied in Ling et al.~\cite{LNQY09}.
Recently, Kellner~et al.~\cite{KTT15} have important work on positive maps and
the set containment problem, and
have proposed semidefinite relaxation methods.

Positive maps have applications in Mechanics.
In elasticity theory, an elasticity tensor can be represented
by an array $\mathcal{B}$ as in \reff{nt:yPhi(xx)y},
which determines the linear map $\Phi$ as in \reff{nt:yPhi(xx)y}.
It is said to satisfy the Legendre-Hadamard condition \cite{C91} if
\[
B(x,y)  \geq 0 \quad
\forall \, x \in \re^p, \, \forall \, y \in \re^q.
\]
Such map is also said to be elliptic.
Moreover, the elasticity tensor is said to be strongly elliptic if
$B(x,y)>0$ for all $x \ne 0$ and $y \ne 0$. Clearly,
the Legendre-Hadamard condition is satisfied if and only if
the associated linear map is positive.
Similarly, it is strongly elliptic if and only if $B(x,y)$ is strictly positive
on the unit spheres $\|x\|_2 = \|y\|_2=1$
($\| \cdot \|_2$ denotes the standard $2$-norm.)
The Legendre-Hadamard condition and strong ellipticity
play important roles in elasticity theory. We refer to \cite{A83,BB02,C91}
and the references therein.
%
%Furthermore, the strongly ellipticity condition satisfies if and only if
% map $\mathcal{B}$ is a positive map with strictly inequality.
%To check the strongly ellipticity condition,
%a direct algorithm was proposed for $p=q=2$ in \cite{QDH09}.
%Some approximation solution and approximation bounds were discussed in \cite{LNQY09}
% and a practical numerical algorithm was proposed in \cite{WQZ09}.
%

\subsection{Separable matrices}
\label{ssc:sepmat}

The cone dual to the positive map cone $\mathscr{P}^{p,q}$ also has important applications.
It is the cone of so-called separable matrices.
For two matrices $B \in \mc{S}^p$ and $C \in \mc{S}^q$,
$B \otimes C$ denotes their Kronecker product,
i.e., $B \otimes C$ is the block matrix
\[
B \otimes C :=  \big( B_{ik} C \big)_{1 \leq i, k \leq p}.
\]
Let $\mathscr{K}^{p,q}$ be the subspace
spanned by all such Kronecker products:
\be \label{matspa:Kpq}
\mathscr{K}^{p,q} = \mbox{span} \left \{
B \otimes C : \,  B \in \mc{S}^p, \, C \in \mc{S}^q
\right\} .
\ee
%
%Each matrix $A \in \mathscr{K}^{p,q}$ is uniquely determined
%by the array $(a_{ijkl})$ with indices
%\[
%1 \leq i \leq j \leq p, \, 1 \leq k \leq l \leq q.
%\]
%
The set $\mathscr{K}^{p,q}$ is a proper subspace of $\mc{S}^{pq}$.
Its dimension is not $p^2q^2(p^2q^2+1)/2$, but instead
\[
\dim  \mathscr{K}^{p,q} =  \frac{1}{4} p(p+1)q(q+1).
\]
Each $A \in  \mathscr{K}^{pq}$ is uniquely determined
by the array
\[
\mc{A} = (a_{ijkl})_{ (i,j,k,l) \in \Omega} \in \re^{\Omega},
\]
in the way that
\be \label{Aindx=aijkl}
A_{(i-1)q + j, (k-1)q+ l} =  a_{ijkl} \quad
\forall\, (i,j,k,l) \in \Omega.
\ee
As in Dahl et al.~\cite{DLMO07}, a matrix $A \in  \mathscr{K}^{pq}$
is said to be {\it separable} if
there exists $B_j \in \mc{S}_+^p, C_j \in \mc{S}_+^q$
($j=1,\ldots, L$) such that
\be \label{A=sum:BotimC}
A  =  B_1 \otimes C_1  + \cdots +  B_L \otimes C_L.
\ee
The equation \reff{A=sum:BotimC} is called an
{\it S-decomposition} of $A$.
Let $\mathscr{S}^{p,q}$ be the cone of all such separable matrices:
\be \label{cone:sepmat}
\mathscr{S}^{p,q} := \Big\{
\sum_{j=1}^L   B_j \otimes C_j : \,
\mbox{ each } B_j \in \mc{S}_+^p, C_j \in \mc{S}_+^q, L \in \N
\Big\}.
\ee
The cones $\mathscr{S}^{p,q}$ and $\mathscr{P}^{p,q}$
are dual to each other (cf.~Prop.~\ref{prop:cone:posep}).
%
%As introduced in \cite{DLMO07},
%a  matrix $A \in  \mc{S}_+^{pq}$ of unit trace are called density matrix,
%or mixed quantum state.
%

In quantum information theory, an important problem is to check
if a quantum system is separable or entangled (cf.~\cite{DLMO07}).
A quantum system can be represented by a density matrix,
which is positive semidefinite and has trace one. Thus, a quantum system
is separable (resp., entangled) if its density matrix
is separable (resp., not separable).
%
%For simplicity, we are concerned in this paper with two subsystems  in real space.
%
Checking whether or not a density matrix is separable
needs to detect the separability/entanglement. To do this,
approximation methods were proposed in \cite{DLMO07,QDH09},
by solving a sequence of bi-quadratic optimization problems.
%%
%%On the other hand, the entanglement problem can be reformulated in terms of positive  maps,
%%stated in \cite{PR03}. That is, $A$ is separable if and only if $tr(A\Phi)\geq 0$
%%for all $\Phi\in\mathscr{S}^{p,q}$, and $\mathscr{S}^{p,q}$
%%are called entanglement witnesses in field of quantum information theory,
%%see \cite{D13} and references therein.
%%
Typically, it is difficult to check separability.
Indeed, the weak membership problem for separable matrices
is NP-hard, as shown by Gurvits~\cite{Gurv}.

\subsection{Contributions}

In this paper, we propose new methods for checking positive maps and
separable matrices.

Checking positivity of a linear map $\Phi$ is equivalent to
checking nonnegativity of the associated bi-quadratic form $B(x,y)$.
So, Lasserre's hierarchy of semidefinite relaxations (cf.~\cite{Las01})
can be applied to solve the question.
Under some optimality conditions,
Lasserre's hierarchy was proved to have finite convergence (cf.~\cite{Nie-opcd}).
For convex polynomial optimization, Lasserre's hierarchy
also has finite convergence, under the strict convexity or sos-convexity assumption
(cf.~\cite{dKlLau11,Las09}). An improvement of Lasserre type relaxations
is proposed in \cite{LTY15}.
For checking positive maps, a sufficient criteria
was given in \cite{KTT15}; some convex relaxations were proposed in \cite{LNQY09}.
Such earlier existing relaxations may not be tight
for checking positivity of some linear maps. For this reason,
this paper proposes a new hierarchy of
semidefinite relaxations (cf.~\S\ref{sc:pm}).
We prove the following property for it:
for {\it every} linear map, its positivity
can be detected by solving a finite number of semidefinite relaxations
contained in this new hierarchy.
For checking positivity of linear maps,
this is the first type of semidefinite relaxations
possessing the aforementioned property,
to the best of the authors' knowledge.

Checking separability of a matrix is equivalent to checking
whether or not it has an S-decomposition as in \reff{A=sum:BotimC}.
For recent work on entanglement or separability,
we refer to \cite{DLMO07,DPS04,Gurv,GB02}.
%
%There exists earlier work of applying semidefinite relaxations
%to check positive maps or separability (cf.~\cite{DPS04,KTT15,LNQY09}).
%
Most earlier existing work can detect inseparability
if the matrix is not separable. However, if the matrix is separable,
these work usually cannot detect the separability,
because an S-decomposition is often lacking. To check separability,
we show that the question is equivalent to
a truncated moment problem with special structures.
To solve it, we construct a hierarchy of semidefinite relaxations.
If the matrix is not separable, we can get a certificate for that.
If it is, we can get an S-decomposition.
To the best of the authors' knowledge, this is the first work
that possesses this property.

The paper is organized as follows. Section~\ref{sc:prlm} presents
some preliminaries in the field of polynomial optimization, moments,
and duality of positive maps and separable matrices.
Section~\ref{sc:pm} discusses how to check if a map is positive or not.
Section~\ref{sc:sep} discusses how to check whether a matrix is separable or not.
Last, we present some numerical examples in Section~\ref{sc:nmexp}.

\section{Preliminaries}
\label{sc:prlm}
\setcounter{equation}{0}

\noindent
{\bf Notation}
The symbol $\N$ (resp., $\re$, $\cpx$) denotes the set of
nonnegative integral (resp., real, complex) numbers.
Let $p,q$ be positive integers.
Denote the variables
\[
x:=(x_1,\ldots,x_p), \quad y:=(y_1,\ldots,y_q).
\]
Denote the $p$-dimensional vector of all ones by $\mathbf{1}_p$.
For convenience, denote
\[
(x,y) = (x_1,\ldots,x_p, \, y_1,\ldots,y_q).
\]
Let $\mbM[x,y]$ be the set of all monomials in $(x,y)$ and
\[
\re[x,y] := \re[x_1,\ldots,x_p,y_1,\ldots,y_q]
\]
be the ring of real polynomials in $(x,y)$.
For $d>0$, $\mbM[x,y]_d$ (resp., $\re[x,y]_d$) denotes
the set of all monomials (resp., polynomials) with degrees at most $d$.
For a set $F \subseteq \re[x,y]$ and a pair $(u,v) \in \re^p \times \re^q$,
the notation
\[
[(u,v)]_{F}
\]
denotes the vector of all polynomials in $F$ evaluated at the point $(u,v)$.
In particular, denote
\be \label{nt:[(u,v)]d}
[(u,v)]_{d} := [(u,v)]_{\mbM[x,y]_d}.
\ee
%
%Let
%\[
%\N^{p+q}_d :=\{ \af \in \N^{p+q} : \af_1+\cdots+\af_n \leq p+q \}.
%\]
%
It can be counted that the dimension of the vector $[(u,v)]_d$ is
$\binom{p+q+d}{d}$, the cardinality of the set $\mbM[x,y]_d$.
For $t$, $\lceil t \rceil$ denotes
the smallest integer that is greater than or equal to $t$.
%
%For a positive integer $k$, denote $[k]:=\{1,2,\ldots, k\}$.
%
%For $x \in \re^n$, $[x]_{d}$ denotes the vector
%\[
%[x]_{d} := \bbm 1 & x_1 &\cdots & x_n & x_1^2 & x_1x_2 & \cdots & x_n^{d}\ebm^T.
%\]
%The superscript $^T$ denotes the transpose of a matrix or vector.
%For a measure $\mu$, $\supp{\mu}$ denotes its support.
%

\subsection{Sum of squares and positive polynomials}

Let $h:=(h_1,\ldots,h_s)$ be a tuple of $s$ polynomials in $\re[x,y]$.
Denote by $I(h)$ the ideal generated by $h$:
\[
I(h) = h_1 \cdot \re[x,y] + \cdots + h_s  \cdot \re[x,y].
\]
In practice, we need to work with a finitely dimensional
subspace in $I(h)$. We denote the
$N$-th {\it truncation} of $I(h)$ as
\be \label{df:<h>k}
I_{N}(h) :=
h_1 \cdot \re[x,y]_{N-\deg(h_1)} + \cdots + h_s  \cdot \re[x,y]_{N-\deg(h_s)}.
\ee
A polynomial $\sig$ is said to be sum of squares (SOS)
if $\sig = f_1^2+\cdots+ f_k^2$ for some real polynomials $f_1,\ldots, f_k$.
The set of all SOS polynomials in $(x,y)$ is denoted as $\Sig[x,y]$.
For a degree $D$, denote the truncation
\[
\Sig[x,y]_D := \Sig[x,y] \cap \re[x,y]_D.
\]
It is a closed convex cone for all even $D>0$.
The symbol $int(\Sig[x,y]_D)$ denotes the interior of $\Sig[x,y]_D$.
For a tuple $g:=(g_1,\ldots,g_t)$ of polynomials in $\re[x,y]$,
the {\it quadratic module} generated by $g$ is the set
\be \label{df:Q(g)}
Q(g):=  \Sig[x,y] + g_1 \cdot \Sig[x,y] + \cdots + g_t \cdot \Sig[x,y].
\ee
The $k$-th truncation of $Q(g)$ is the set
\be \label{Qk(g):def}
Q_k(g):=
\Sig[x,y]_{2k} + g_1 \cdot \Sig[x,y]_{2k - \deg(g_1)}
+ \cdots + g_t \cdot \Sig[x,y]_{2k - \deg(g_t)}.
\ee

Let $h$ and $g$ be the polynomial tuples as above. Consider the set
\be \label{df:S(hg)}
S = \{ (u,v) \in \re^p \times \re^q: \, h(u,v) = 0, \, g(u,v) \geq 0\}.
\ee
Clearly, if $f \in I(h) + Q(g)$,
then $f$ is nonnegative on $S$.
Interestingly, the reverse is also true under some general conditions.
The set $I(h) + Q(g)$ is called {\it archimedean}
if there exists $\phi \in I(h) + Q(g)$ such that
$\phi(x,y) \geq 0$ defines a compact set in the space $\re^p \times \re^q$.
When $I(h) + Q(g)$ is archimedean, Putinar \cite{Put} proved that
if $f \in \re[x,y]$ is positive on $S$ then $f  \in I(h) + Q(g)$.
Moreover, as shown recently in \cite{Nie-opcd},
if $f$ is nonnegative on $S$
and satisfies some general optimality conditions,
then we also have $f  \in I(h) + Q(g)$.
We refer to Lasserre's book \cite{LasBok}
and Laurent's survey \cite{Lau},
for additional information on polynomial optimization.

\subsection{Truncated moment problems}

Let $\re^{\mbM[x,y]_d}$ be the space of vectors indexed
by monomials in the set $\mbM[x,y]_d$.
A vector in $\re^{\mbM[x,y]_d}$ is called a
{\it truncated multi-sequence} (tms) of degree $d$.
For a tms $w \in \re^{\mbM[x,y]_d}$, we can index it as
\[
w = ( w_{ x^\af y^{\bt} } ) _{ x^\af y^{\bt} \in \mbM[x,y]_d} .
\]
Define the scalar product between $\re[x,y]_d$ and $\re^{\mbM[x,y]_d}$ such that
\be \label{df:<f,w>}
\Big \langle \sum_{ |\af| + |\bt| \leq d  }
c_{\af,\bt} x^\af y^{\bt} , w  \Big \rangle :=
\sum_{ |\af| + |\bt| \leq d  }
c_{\af,\bt} w_{ x^\af y^{\bt} },
\ee
where $c_{\af,\bt}$ are the coefficients.
%
% define the linear functional $\mathscr{L}_w$
%on $\re[z]_d$ such that
%\[
%\mathscr{L}_w \big( a \big) :=   w_a
%\quad \forall \, a \in \mbM[z]_d.
%\]
%
The tms $w$ is said to admit a representing measure
whose support is contained in a set
$T$ if there exists a Borel measure $\mu$
supported in $T$ (i.e., $\supp{\mu} \subseteq T$) such that
\[
w_a = \int  a \, \mt{d} \mu
\quad \forall \, a \in \mbM[x,y]_d.
\]
If so, such $\mu$ is called a $T$-representing measure for $w$
and we say that $w$ admits the measure $\mu$.
An interesting question is how to check whether a tms
admits a $T$-representing measure or not.
The method in \cite{Nie-ATKMP} can be applied to do this.
Note that this problem is not polynomial optimization.
The classical Lasserre's relaxations in \cite{Las01} for polynomial optimization
is not very suitable for solving the question.

Let $\theta \in \re[x,y]_{2k}$ with $\deg(\theta) \leq 2k$.
The $k$-th {\it localizing matrix} of $\theta$,
generated by $w\in \re^{\mbM[x,y]_{2k}}$,
is the symmetric matrix $L_{\theta}^{(k)}(w)$ satisfying
(see \reff{df:<f,w>} for $\langle, \rangle$)
\[
vec(f_1)^T \Big( L_{\theta}^{(k)}(w) \Big) vec(f_2)
= \langle \theta f_1 f_2, w \rangle
\]
for all $f_1, f_2 \in \re[x,y]$ with
\[ \deg(f_1), \deg(f_2) \leq k - \lceil \deg(\theta)/2 \rceil.\]
In the above, $vec(f_i)$ denotes the coefficient vector of the polynomial $f_i$.
When $\theta = 1$ (the constant polynomial $1$),
$L_1^{(k)}(w)$ is called a {\it moment matrix} and is denoted as
\be \label{df:Mk(w)}
M_k(w):= L_{1}^{(k)}(w).
\ee
The columns and rows of $L_{\theta}^{(k)}(w)$, as well as $M_k(w)$,
are indexed by monomials $a \in \mbM[x,y]$ with $ \deg(\theta a^2) \leq 2k$.

Let $S$ be as in \reff{df:S(hg)}.
If $w$ admits an $S$-representing measure,
then (cf.~\cite{CuFi05,Nie-ATKMP})
\be \label{Lh=0Lg>=0}
L_{h_i}^{(k)}(w) = 0 \, (1 \leq i \leq s), \quad
L_{g_j}^{(k)}(w) \succeq 0 \, (1 \leq j \leq t), \quad
M_k(w)\succeq 0.
\ee
The reverse is typically not true. For convenience, denote
\be \label{df:L(h):L(g)}
\left\{ \baray{l}
L_h^{(k)}(w) = \Big(L_{h_1}^{(k)}(w), \ldots, L_{h_s}^{(k)}(w) \Big),  \\
L_g^{(k)}(w) = \Big(L_{g_1}^{(k)}(w), \ldots, L_{g_t}^{(k)}(w) \Big).
\earay \right.
\ee
In the above, $\mbox{diag}(X_1, \ldots, X_r)$ denotes the block diagonal matrix
whose diagonal blocks are $X_1, \ldots, X_r$.\
Let
$
d_0 = \max\, \{1,  \lceil \deg(h) /2 \rceil,  \lceil \deg(g) /2 \rceil  \}.
$
If $w$ satisfies \reff{Lh=0Lg>=0} and
\be \label{con:fec}
\rank \, M_{k-d_0}(w) \, = \, \rank \, M_k(w),
\ee
then $w$ admits an $S$-representing measure (cf.~\cite{CuFi05,Nie-ATKMP}).
When \reff{Lh=0Lg>=0} and \reff{con:fec} hold,
the tms $w$ admits a unique representing measure $\mu$ on $\re^n$;
moreover, the measure $\mu$ is supported on $r:=\rank\, M_k(w)$ distinct points in $S$.
The points in $\supp{\mu}$ can be found by
solving some eigenvalue problems~\cite{HenLas05}.
For convenience, we say that $w$ is {\it flat with respect to}
$h=0$ and $g\geq 0$
if \reff{Lh=0Lg>=0} and \reff{con:fec} are {\it both} satisfied.

For two tms' $w \in \re^{ \mbM[x,y]_{2k} }$ and
$z \in \re^{ \mbM[x,y]_{2l} }$ with $k<l$,
we say that $w$ is a truncation of $z$,
or equivalently, $z$ is an extension of $w$,
if $w_a = z_a$ for all $a \in \mbM[x,y]_{2k}$.
Denote by $z|_d$ the subvector of $z$
whose entries are indexed by $a \in \mbM[x,y]_d$.
Thus, $w$ is a truncation of $z$ if $z|_{2k} = w$.
Throughout the paper, if $z|_{2k} = w$ and $w$ is flat,
we say that $w$ is a flat truncation of $z$.
Similarly, if $z|_{2k} = w$ and $z$ is flat,
we say that $z$ is a flat extension of $w$.
Flat extensions and flat truncations are proper criteria
for checking convergence of Lasserre's hierarchies
in polynomial optimization (cf.~\cite{Nie-ft}).

\subsection{Properties of $\mathscr{P}^{p,q}$ and $\mathscr{S}^{p,q}$ }

The positive map cone $\mathscr{P}^{p,q}$ and
the separable matrix cone $\mathscr{S}^{p,q}$
can be thought of as subsets of the vector space $\re^{\Omega}$,
for $\Omega$ as in \reff{df:Omg}.
For $\mc{B} \in \mathscr{P}^{p,q}$ and $\mc{A} \in \mathscr{S}^{p,q}$,
we can index them as
\[
\mc{B} = (b_{ijkl} )_{ (i,j,k,l) \in \Omega }, \quad
\mc{A} = (a_{ijkl} )_{ (i,j,k,l) \in \Omega }.
\]
Define their inner product in the standard way as
\[
\langle \mc{A}, \mc{B} \rangle := \sum_{(i,j,k,l) \in \Omega}
a_{ijkl} b_{ijkl}.
\]
The standard definition of dual cones is used in the paper.
A cone $\mc{C}$ is said to be pointed if
$\mc{C} \cap -\mc{C} = \{0\}$, and
it is said to be solid if it has nonempty interior.

\begin{prop} \label{prop:cone:posep}
The cones $\mathscr{P}^{p,q}$ and $\mathscr{S}^{p,q}$
are proper (i.e., closed, convex, pointed, and solid),
and they are dual to each other, i.e.,
\be \label{dual:P=S*}
(\mathscr{P}^{p,q} )^* = \mathscr{S}^{p,q}, \quad
(\mathscr{S}^{p,q} )^* = \mathscr{P}^{p,q}.
\ee
\end{prop}

The convexity of $\mathscr{P}^{p,q}$ are straightforward.
As in \cite[Theorem 2]{DLMO07}, it holds that
\[
\mathscr{S}^{p,q}=\mbox{conv}
\{(x\otimes y)(x\otimes y)^T| x\in\mathbb{R}^p, y\in\mathbb{R}^q\}.
\]
(The $\mbox{conv}$ denotes the convex hull.)
So, we can get the convexity of $\mathscr{S}^{p,q}$.
The polynomial $(x^Tx)(y^Ty) \in \mathscr{P}^{p,q}$
is strictly positive on the bi-sphere $\|x\|_2=\|y\|_2=1$.
The set $\mathscr{S}^{p,q}$ is the cone of
truncated multi-sequences in $\re^{\Omega}$
that admit representing measures supported on
the bi-sphere $\|x\|_2=\|y\|_2=1$.
Hence, the closedness, pointedness, and solidness of
the cones $\mathscr{P}^{p,q}$ and $\mathscr{S}^{p,q}$,
as well as the duality relationship \reff{dual:P=S*},
can be implied by \cite[Prop. 3.2]{linmomopt}.
We refer to \cite{GB02} for related work on
positive maps and separable matrices.

%
%Furthermore, the duality relationship \reff{dual:P=S*} is implied, see also \cite{GB02} and \cite[Prop. 3.2]{linmomopt}.
%Note that the bi-quadratic form $(x^Tx)(y^Ty)$
%is strictly positive over the unit spheres $\|x\|_2 = \|y\|_2=1$.
%The solidness of $\mathscr{P}^{p,q}$ and $\mathscr{S}^{p,q}$ can be obtained and the proof is completed here.
%
%and the duality relationship
%can be implied by \cite{GB02}.
%%%%%%%%%%%%%%%%%%%%%
\iffalse
Third, we prove the dual relationship. Clearly, it holds that
\[
 \mathscr{S}^{p,q} \subseteq (\mathscr{P}^{p,q} )^*.
\]
Suppose otherwise they are not equal, then there exists
$A \in \mathscr{K}^{p,q}$ such that
\[
 A \in (\mathscr{P}^{p,q} )^*, \quad
 A \not\in \mathscr{S}^{p,q}.
\]
The set $\mathscr{S}^{p,q}$ is a closed convex cone.
By the strict separating hyperplane theorem,
there exists $P \in \mathscr{K}^{p,q}$
such that
\[
\langle  A , P \rangle < 0,  \quad
\langle  X , P \rangle \geq 0 \quad  \forall  \, X \in\mathscr{S}^{p,q}.
\]
The latter implies that $P \in \mathscr{P}^{p,q}$. However,
the former implies that $A \in (\mathscr{P}^{p,q})^*$.
This results in a contradiction.
So, the dual relation \reff{dual:P=S*} must be true.
\fi
%%%%%%%%%%%%%%%%%%%%%%%%%%%%

\section{Checking positive maps}
\label{sc:pm}
\setcounter{equation}{0}

This section discusses how to check whether a linear map
$\Phi: \mc{S}^p \to \mc{S}^q$ is positive or not.
The linear map $\Phi$ is uniquely determined by
\be \label{bqfm:B(xy)}
B(x,y) := y^T \Phi(xx^T) y,
\ee
a bi-quadratic form
in $x:=(x_1, \ldots, x_p)$ and $y:=(y_1,\ldots, y_q)$.
To check the positivity of $\Phi$, it is equivalent to
determine whether or not $B(x,y)$ is nonnegative on $x^Tx=y^Ty=1$.
So, we consider the optimization problem
\be  \label{minB(xy):2sph}
\left\{\baray{rl}
b_{min} := \min & B(x,y) \\
 s.t. & x^Tx =1,\, y^Ty=1.
\earay \right.
\ee
The first order optimality condition for \reff{minB(xy):2sph} implies that
\be \label{kkt:Bxy}
\bbm  B_x(x,y) \\ B_y(x,y)  \ebm =
\bbm  2\lmd_1 x \\ 2\lmd_2 y \ebm.
\ee
In the above, $B_x(x,y)$ (resp., $B_y(x,y)$) denotes the gradient of
$B(x,y)$ in $x$ (resp., $y$). Since $B(x,y)$ is a quadratic form
in both $x$ and $y$, it holds that
\be \label{xBx=yBy=2B}
\bbm  x^TB_x(x,y) \\ y^TB_y(x,y)  \ebm =
\bbm  2B(x,y) \\ 2B(x,y) \ebm.
\ee
Thus, \reff{kkt:Bxy} and \reff{xBx=yBy=2B} imply that
\[
\lmd_1 = \lmd_2 = B(x,y).
\]
Note that $(x^*,y^*)$ is optimal for \reff{minB(xy):2sph} if and only if
$(\pm x^*, \pm y^*)$ are all optimal.
By choosing the right signs, \reff{minB(xy):2sph} always has
an optimizer $(x^*,y^*)$ satisfying
\[
\mathbf{1}_p^Tx^* \geq 0, \, \mathbf{1}_q^Ty^* \geq 0.
\]
Therefore, \reff{minB(xy):2sph} is equivalent to the optimization problem
\be \label{minB(x,y):kkt}
\left\{\baray{rl}
\min & B(x,y) \\
 s.t. & x^Tx =1,\, y^Ty=1, \\
 &  B_{x}(x,y) - 2B(x,y) x = 0, \\
 &  B_{y}(x,y) - 2B(x,y) y = 0, \\
 &  \mathbf{1}_p^Tx \geq 0, \, \mathbf{1}_q^Ty \geq 0.
\earay \right.
\ee
It is a polynomial optimization problem of degree $5$.
Compared with \reff{minB(xy):2sph}, the problem
\reff{minB(x,y):kkt} has two main advantages:
\bit

\item The problem \reff{minB(x,y):kkt} has two more equalities
\[
B_x(x,y)-2B(x,y)x = 0, \quad B_y(x,y) - 2B(x,y) y = 0.
\]
By using them, Lasserre's hierarchy of semidefinite relaxations
(see \reff{min<B,w>:mom}) has finite convergence.
This is shown in Theorem~\ref{thm:posmap:cvg}.
However, without using them, Lasserre's hierarchy of
semidefinite relaxations for solving (3.2)
directly may not have finite convergence.

\item  The problem \reff{minB(x,y):kkt} has two more inequalities.
The number of minimizers of \reff{minB(x,y):kkt} is
only one quarter of those of \reff{minB(xy):2sph}.
Thus, in computations (e.g., by software {\tt GloptiPoly~3} \cite{GloPol3}),
solving \reff{minB(x,y):kkt} is often much easier
than solving \reff{minB(xy):2sph}, for numerical reasons.
This is because it is easier for {\tt GloptiPoly~3}
to identify convergence by using the flat extension condition
(see \reff{flatcd:Mt(w*)}).

\eit

The optimal value $b_{min}$ of \reff{minB(x,y):kkt}
is equal to that of \reff{minB(xy):2sph}.
Let $h,g$ be the tuples of constraining polynomials in \reff{minB(x,y):kkt}:
\be \label{df:hg:Bkkt}
\left\{\baray{l}
h = \Big(  x^Tx-1 , y^Ty-1, \,
B_{x}(x,y)  - 2B(x,y) x, \, B_{y}(x,y)  - 2B(x,y) y \Big), \\
g = \big(\mathbf{1}_p^Tx, \, \mathbf{1}_q^Ty \big).
\earay \right.
\ee
Lasserre's hierarchy \cite{Las01} of semidefinite relaxations for
solving \reff{minB(x,y):kkt} is
\be \label{min<B,w>:mom}
\left\{\baray{rl}
b_k^{(1)} := \min & \langle B, w \rangle  \\
 s.t. &  \langle 1, w \rangle = 1, \, L_{h}^{(k)}(w) = 0, \\
  & M_k(w) \succeq 0, \, L_{g}^{(k)}(w)  \succeq 0, \\
  &  w \in \re^{ \mathbb{M}[x,y]_{2k} },
\earay \right.
\ee
for the orders $k = 3, 4, \ldots$.
The product $\langle, \rangle$ is defined as in \reff{df:<f,w>}.
We refer to \reff{df:Mk(w)} and \reff{df:L(h):L(g)}
for matrices $M_k(w)$, $L_{h}^{(k)}(w)$, and $L_{g}^{(k)}(w)$.
They are all linear in $w$.
The dual problem of \reff{min<B,w>:mom} is
\be \label{max:B-gm:sos}
\left\{\baray{rl}
b_k^{(2)} := \max & \gamma \\
 s.t. & B - \gamma \in  I_{2k}(h) + Q_k(g).
\earay \right.
\ee
In the above, the notation $I_{2k}(h)$ and $Q_k(g)$
are respectively defined as in \reff{df:<h>k} and \reff{Qk(g):def}.
By the weak duality, it holds that for all $k$
\be \label{bk<=bmin}
b_k^{(2)} \leq b_k^{(1)} \leq b_{min}.
\ee
As in \cite{Las01},
$\{ b_k^{(1)} \}$ and $\{ b_k^{(2)} \}$
are both monotonically increasing.

%
%\begin{remark} \label{rmk:ft:posmap}
%

A practical question is how to check the convergence of
$b_k^{(1)}$ and $b_k^{(2)}$ to $b_{min}$.
The following rank condition, for some $t \in [2,k]$,
\be \label{flatcd:Mt(w*)}
\rank \, M_{t}(w^*)  = \rank \, M_{t+1} (w^*)
\ee
is a proper stopping criterion (cf.~\cite{HenLas05,Nie-ft}).
If \reff{flatcd:Mt(w*)} is satisfied, then $b_k^{(1)}  = b_{min}$
and we can get $r:=\rank \,M_{t}(w^*)$ global minimizers of
\reff{minB(x,y):kkt}. This can be seen as follows. From
\reff{flatcd:Mt(w*)}, by Theorem~1.1 of \cite{CuFi05}
(also see \cite{HenLas05,Nie-ft} for elaborations),
we can get the decomposition
\[
w^*|_{2t} = c_1 [(u_1,v_1)]_{2t} + \cdots + c_r [(u_r,v_r)]_{2t},
\]
where each $c_i>0$ and $u_i^Tu_i = v_i^Tv_i = 1$.
The equality $\langle 1, w^* \rangle=1$ leads to
\[
c_1 + \cdots + c_r  = 1.
\]
Since $w^*$ is an optimizer of \reff{min<B,w>:mom},
the above decomposition of $w^*|_{2t}$ implies
\[
b_k^{(1)} = c_1 B(u_1,v_1) + \cdots + c_r B(u_r,v_r).
\]
Since $b_{min} \leq B(u_i,v_i)$ for each $i$, \reff{bk<=bmin} shows that
\[
b_k^{(1)} \leq B(u_1,v_1), \ldots, b_k^{(1)} \leq B(u_r,v_r).
\]
By the above, we can get that
\[
b_{min} \leq B(u_1,v_1) = \cdots = B(u_r,v_r) =  b_k^{(1)} \leq b_{min}.
\]
So, $b_k^{(1)} = b_{min}$, and $(u_1,v_1), \ldots, (u_r,v_r)$
are global minimizers of \reff{minB(x,y):kkt}.

\begin{alg} \label{alg:posmap}
(Check positivity of a linear map $\Phi: \mc{S}^p \to \mc{S}^q$.)
Formulate the bi-quadratic form $B(x,y)$ as in \reff{bqfm:B(xy)}.
Let $k :=3$.
\bit

\item [Step 1] Solve the semidefnite relaxation \reff{min<B,w>:mom}
for a minimizer $w^{*,k}$.

\item [Step 2]
If \reff{flatcd:Mt(w*)} is satisfied for some $t\in [2,k]$, go to Step~3;
otherwise, let $k:=k+1$ and go to Step~1.

\item [Step 3] Compute $r:=\rank \,M_{t}(w^*)$ global minimizers for
\reff{minB(x,y):kkt}. Output $b_k^{(1)}$ as the minimum value $b_{min}$
of \reff{minB(xy):2sph}. If $b_{min} \geq 0$, then $\Phi$
is a positive map; otherwise, it is not.

\eit

\end{alg}

In Step~3, the method in \cite{HenLas05} can be
applied to get global minimizers for \reff{minB(x,y):kkt}.
The convergence of Algorithm~\ref{alg:posmap} is summarized as follows.

\begin{theorem} \label{thm:posmap:cvg}
Let $B(x,y)$ be the bi-quadratic form
for a linear map $\Phi: \mc{S}^p \to \mc{S}^q$
as in \reff{bqfm:B(xy)}, and let
$b_{min}$ be the optimal value of \reff{minB(xy):2sph}.
Let $b_k^{(1)}, b_k^{(2)}$ be the optimal values as in
\reff{min<B,w>:mom}-\reff{max:B-gm:sos}.
Then we have:

\bit

\item [(i)] For all $k$ sufficiently large, it holds that
\[
b_k^{(1)} = b_k^{(2)}  = b_{min}.
\]
Hence, $\Phi$ is positive if and only if
$b_k^{(1)} \geq 0$ (or $b_k^{(2)} \geq 0$) for some $k$.

\item [(ii)] Assume \reff{minB(xy):2sph} has finitely many minimizers.
If $k$ is large enough, then for every optimizer
$w^*$ of \reff{min<B,w>:mom} there exists $t \in [2,k]$
satisfying \reff{flatcd:Mt(w*)}.

\eit
\end{theorem}
\begin{proof}
(i) The optimality condition \reff{kkt:Bxy} is equivalent to that
\[
\rank \, \widetilde{B}(x,y) = 2, \quad \mbox{ where } \quad
\widetilde{B}(x,y) :=
\bbm B_x(x,y)  &  x  & 0 \\ B_y(x,y)  & 0 & y \ebm.
\]
Let $\phi_1, \ldots, \phi_J$ be the all $3$-by-$3$ minors
of $\widetilde{B}(x,y)$ and $\widetilde{h}$ be the tuple
\[
\widetilde{h}  := (x^Tx -1,y^Ty-1, \phi_1, \ldots, \phi_J).
\]
Then \reff{minB(xy):2sph} is equivalent to the optimization problem
\be \label{mB:th(xy)==0}
\min \quad B(x,y) \quad s.t. \quad \widetilde{h}(x,y) = 0.
\ee
Lasserre's hierarchy of semidefinite relaxations
for solving \reff{mB:th(xy)==0} is
\be
\left\{\baray{rl}
\widetilde{b}_k^{(1)}  := \min & \langle B, w \rangle  \\
 s.t. & \langle 1, w \rangle = 1, \, L_{\tilde{h}}^{(k)}(w) = 0,  \\
  &   \, M_k(w) \succeq 0, \, w \in \re^{ \mathbb{M}[x,y]_{2k} },
\earay \right.
\ee
for $k=3,4,\ldots$. Its dual optimization problem is
\be
\left\{\baray{rl}
\widetilde{b}_k^{(2)} := \max & \gamma \\
 s.t. & B - \gamma \in  I_{2k}(\widetilde{h}) + \Sig[x,y]_{2k}.
\earay \right.
\ee
By Theorem~2.3 of \cite{Nie-jac}, for all $k$ big enough, we have
\[
\widetilde{b}_k^{(1)} = \widetilde{b}_k^{(2)} = b_{min}.
\]
That is, both $\{ \widetilde{b}_k^{(2)} \}$
and $\{ \widetilde{b}_k^{(1)} \}$ have finite convergence to $b_{min}$.
Consider the optimization problem
\be  \label{minB:h(x,y)=0}
\min \quad B(x,y) \quad s.t. \quad h(x,y) = 0.
\ee
Lasserre's hierarchy of semidefinite relaxations for \reff{minB:h(x,y)=0} is
\be  \label{mom:miB:h(z)=0}
\left\{\baray{rl}
\widehat{b}_k^{(1)} := \min & \langle B, w \rangle  \\
 s.t. & \langle 1, w \rangle = 1, \, L_{h}^{(k)}(w) = 0, \\
 &  M_k(w) \succeq 0, w \in \re^{ \mathbb{M}[x,y]_{2k} }.
\earay \right.
\ee
Its dual optimization problem is
\be \label{sos:miB:h(z)=0}
\left\{\baray{rl}
\widehat{b}_k^{(2)} := \max & \gamma \\
 s.t. & B - \gamma \in  I_{2k}(h)  + \Sig[x,y]_{2k}.
\earay \right.
\ee
The feasible sets of \reff{mB:th(xy)==0}
and \reff{minB:h(x,y)=0} are same.
By Theorem~3.1 of \cite{Nie-poprv}, the sequence
$\{ \widehat{b}_k^{(2)} \}$ also has finite convergence to $b_{min}$.
Since $\Sig[x,y]_{2k} \subseteq Q_k(g)$, we have
\[
\widehat{b}_k^{(2)} \leq b_k^{(2)} \leq b_k^{(1)} \leq b_{min}
\]
for all $k$. Hence, both $\{ b_k^{(1)} \}$ and $\{ b_k^{(2)} \}$
have finite convergence to $b_{min}$. Thus, by \reff{bk<=bmin},
$\Phi$ is positive if and only  if
for some $k$, $b_k^{1)}\geq 0$ or $b_k^{(2)} \geq 0$.

(ii) In the above, we have shown that
$\{ b_k^{(1)} \}$ and $\{ \widehat{b}_k^{(1)} \}$
have finite convergence to $b_{min}$. For $k$ sufficiently large,
\[
\langle B, w^* \rangle = b_k^{(1)} = \widehat{b}_k^{(1)} = b_{min}.
\]
Because the feasible set of \reff{min<B,w>:mom}
is contained in that of \reff{mom:miB:h(z)=0},
$w^*$ is also a minimizer of \reff{mom:miB:h(z)=0} when $k$ is big enough.
Note that, for large $k$,
\[
\widehat{b}_k^{(1)}  = \widehat{b}_k^{(2)}  =   b_{min}.
\]
Since $\widehat{b}_k^{(1)}, \widehat{b}_k^{(2)}$ are respectively
the optimal values of \reff{mom:miB:h(z)=0}-\reff{sos:miB:h(z)=0},
there is no duality gap between \reff{mom:miB:h(z)=0}
and \reff{sos:miB:h(z)=0}, when $k$ is large. Let
\[
d_h := \max(1, \lceil \deg(h) /2 \rceil ).
\]
Note that \reff{mom:miB:h(z)=0}-\reff{sos:miB:h(z)=0}
are relaxations for solving \reff{minB:h(x,y)=0}.
By the assumption in the item ii), we know that
\reff{minB:h(x,y)=0} has finitely many optimizers.
By Theorem~2.6 of \cite{Nie-ft}, for $k$ big enough,
there exists $t \in [2,k]$ such that
\[
\rank \, M_{t}(w^*)  = \rank \, M_{t+d_h} (w^*).
\]
On the other hand, it always holds that
\[
\rank \, M_{t}(w^*)  \leq \rank \, M_{t+1} (w^*) \leq  \rank \, M_{t+d_h} (w^*).
\]
So, \reff{flatcd:Mt(w*)} must be satisfied
when $k$ is sufficiently large.
\end{proof}

\section{Decomposition of separable matrices}
\label{sc:sep}
\setcounter{equation}{0}

This section discusses how to check whether a matrix is separable or not.
We first formulate the question as a special truncated moment problem,
and then propose a semidefinite algorithm for solving it.

\subsection{An equivalent reformulation}

Recall the matrix space $\mathscr{K}^{p,q}$ as in \reff{matspa:Kpq} and
the separable matrix cone $\mathscr{S}^{p,q}$ as in \reff{cone:sepmat}.
As shown in Dahl~et al.~\cite[Theorem~2.2]{DLMO07}, every separable matrix in
$\mathscr{S}^{p,q}$ is a
%
%{\bf ???nonnegative linear???? may be changed to `` conical'' ???}
%
nonnegative linear combination of rank-$1$ Kronecker products like
\[
(u u^T) \otimes (v v^T),
\]
where $u^Tu=v^Tv=1$.
By choosing the right signs, the above $u,v$ can be chosen such that
\[
\mathbf{1}_p^Tu \geq 0, \quad \mathbf{1}_q^T v \geq 0.
\]
An advantage for using the above inequalities is that
the software {\tt GloptiPoly~3} has better numerical performance
for solving the semidefinite relaxation \reff{min<R,w>},
than the one without such inequalities.
Moreover, using them may help us get shorter S-decompositions.

Denote the set
\be \label{def:S+pq}
K :=  \left\{
(x,y) \in \re^p \times \re^q
\left| \baray{c}
x^Tx = 1, y^Ty = 1, \\
\mathbf{1}_p^Tx \geq 0, \, \mathbf{1}_q^T y \geq 0
\earay \right. \right\}.
\ee
Therefore, $A \in \mathscr{S}^{p,q}$ if and only if
\be \label{A:S-dcmp}
A = \sum_{s=1}^N  c_s (u_s u_s^T) \otimes (v_s v_s^T)
\ee
for $c_1, \ldots, c_N>0$ and $(u_1,v_1),\ldots, (u_N, v_N) \in K$.
The equation \reff{A:S-dcmp} is called an {\it S-decomposition} of $A$.
The above is equivalent to that
\be   \label{A=sum:kron(u_i,v_i)}
A_{\pi(i,j), \pi(k,l)} = \sum_{s=1}^N  c_s \cdot (u_s)_i (v_s)_j  (u_s)_k (v_s)_l
\ee
for all pairs
$(i,j), (k,l) \in [p] \times [q]$, with
\[
\pi(i,j) :=(i-1)q+j,  \quad \pi(k,l) := (k-1)q+l.
\]
Let $\mu$ be the weighted sum of Dirac measures:
\be
\mu := c_1 \dt_{(u_1,v_1)} + \cdots + c_N  \dt_{(u_N, v_N)}.
\ee
Then, \reff{A:S-dcmp} is equivalent to
\[
A_{\pi(i,j), \pi(k,l)} = \int_{ K }
x_iy_j x_k y_l \mathtt{d} \mu  \quad
\forall \,(i,j), (k,l) \in [p] \times [q],
\]
which is then equivalent to that
\be \label{A=int:xy:dmu}
A  = \int_{ K  }
(x x^T) \otimes (y y^T) \mathtt{d} \mu.
\ee
Denote the monomial set
\be \label{df:E}
\mc{E} = \big\{ x_iy_j x_k y_l: \,
1 \leq i \leq k \leq p,   1 \leq j \leq l \leq q \big \}.
\ee
The cardinality of $\mc{E}$ is
\[
\frac{1}{4} p(p+1)q(q+1),
\]
the dimension of the space $\mathscr{K}^{p,q}$.
The monomial $x_iy_j x_k y_l$ can be uniquely identified
by the tuple $(i,j,k,l) \in \Omega$, as in \reff{df:Omg}.
Therefore, we can index each matrix $A \in \mathscr{K}^{p,q}$
equivalently by monomials in $\mc{E}$ as
\[
A_{ x_iy_j x_k y_l }  \, := \,  A_{ \pi(i,j), \pi(k,l) }.
\]
So, each $A \in \mathscr{K}^{p,q}$ can be uniquely identified
by the vector $( A_{ b} )_{ b  \in \mc{E}  }$. Let
\be  \label{a=A|E}
\mathbf{a} \, :=  ( A_{ b } )_{ b \in \mc{E}  }.
\ee
The vector $\mathbf{a}$ is an $\mc{E}$-truncated multi-sequence ($\mc{E}$-tms).
We refer to \cite{Nie-ATKMP}
for such structured truncated moment problems.

If there exists a Borel measure $\mu$ supported in
$K$ satisfying \reff{A=int:xy:dmu}, then $A$ must be separable.
This can be implied by Proposition~3.3 of \cite{Nie-ATKMP}.
Such $\mu$ is called a $K$-representing measure for $\mathbf{a}$.
%
%there exist $\lmd_1, \ldots \lmd_N>0$
%and $(u_1,v_1), \ldots, (u_r,v_r) \in K$
%satisfying \reff{A=sum:kron(u_i,v_i)}, i.e., $A$ is separable.
%

Summarizing the above, we get the proposition.

\begin{prop}  \label{pr:sep=tmp}
For each $A \in \mathscr{K}^{p,q}$,
%
% and $\mathbf{a}$ be as in \reff{a=A|E}.
%
the matrix $A$ is separable (i.e., $A \in \mathscr{S}^{p,q}$)
if and only if \reff{A=int:xy:dmu} is satisfied by a Borel measure
$\mu$ supported in $K$.
\end{prop}

The vector $\mathbf{a}$, as in \reff{a=A|E}, is an $\mc{E}$-tms of degree $4$.
By Proposition~\ref{pr:sep=tmp}, to check if $A$ is separable or not
is equivalent to detecting if $\mathbf{a}$ has a representing measure
supported in $K$. The latter is a truncated moment problem.
Let
\be \label{df:hg:sepmat}
h :=(x^Tx -1,  y^Ty-1), \quad  g := (\mathbf{1}_p^Tx, \mathbf{1}_q^Ty ).
\ee
Suppose $\omega \in \re^{ \mathbb{M}[x,y]_{2t} }$ is
an extension of $\mathbf{a}$, i.e., $\omega|_{\mc{E}} = \mathbf{a}$.
If $\omega$ is flat with respect to
$h=0$ and $g\geq 0$, i.e., it satisfies
\be  \label{cd:flat:omga}
\, L_{h}^{(t)}(\omega) = 0, \quad
 L_{g}^{(t)}(\omega) \succeq 0, \quad
 \rank \, M_{t-1}(\omega) =  \rank \, M_{t}(\omega),
\ee
then there exist $c_i >0$, $(u_i, v_i) \in K$ ($i=1,\ldots,r$) such that
\be \label{omg=sum:cixi}
\omega = c_1 [(u_1, v_1)]_{2t} + \cdots + c_r [(u_r, v_r)]_{2t}.
\ee
The extension condition $\omega|_{\mc{E}} = \mathbf{a}$
and \reff{omg=sum:cixi} imply that
\[
\mathbf{a} = c_1 [(u_1,v_1)]_{\mc{E}} + \cdots + c_r [(u_r,v_r)]_{\mc{E}}.
\]
From \reff{a=A|E}, we can get
\[
A = c_1 (u_1 u_1^T) \otimes (v_1 v_1^T)  + \cdots +
c_r (u_r u_r^T) \otimes (v_r v_r^T).
\]
This gives an S-decomposition for $A$ if $\omega$ is flat.
Such $\omega$ is called a flat extension of $\mathbf{a}$.

If there exists a flat extension of $\mathbf{a}$, then $A$ is separable.
Conversely, if $A$ is separable, then $\mathbf{a}$ must have a flat extension
(cf.~\cite[Prop.~3.3]{Nie-ATKMP}).
When does $\mathbf{a}$ have a flat extension?
If yes, how can we find one?
If no, how do we know its nonexistence?
We propose semidefinite relaxations for solving such questions.

\subsection{A semidefinite algorithm}

By Proposition~\ref{pr:sep=tmp}, a matrix $A \in \mathscr{K}^{p,q}$
is separable if and only if the vector $\mathbf{a}$,
as in \reff{a=A|E}, has a representing measure supported in $K$.
This can be detected by solving semidefinite relaxations.

Choose a generic SOS polynomial $R \in \Sig[x,y]_6$.
Let $h,g$ be as in \reff{df:hg:sepmat}.
For relaxation orders $k \geq 3$, consider the semidefinite relaxation
\be \label{min<R,w>}
\left\{ \baray{rl}
\min & \langle R, w \rangle  \\
s.t. &  w|_{\mc{E}} = \mathbf{a}, \, L_{h}^{(k)}(w) = 0, \,
           w \in \re^{ \mathbb{M}[x,y]_{2k} },  \\
 &  M_k(w) \succeq 0, \, L_{g}^{(k)}(w) \succeq 0.
\earay \right.
\ee
(See \reff{df:<f,w>} for the product $\langle,\rangle$.)
The dual problem of \reff{min<R,w>} is
\be  \label{sos:max<p,y>}
\left\{ \baray{rl}
\max   &  \langle f, \mathbf{a} \rangle  \\
s.t.  &  R - f \in I_{2k}(h) + Q_k( g ), \, f \in \mbox{span}\{\mc{E}\}.
\earay \right.
\ee
The decision variable in \reff{sos:max<p,y>}
is the vector of coefficients of $f$.

\begin{alg} \label{sdpalg:A-tkmp}
(Check membership in the cone $\mathscr{S}^{p,q}$.)
For a given matrix $A \in \mathscr{K}^{p,q}$, do the following:
\bit

\item [Step 0] Choose a generic $R \in \Sig[x,y]_6$. Let $k=3$.

\item [Step 1] If \reff{min<R,w>} is infeasible, then $A$ is not separable and stop;
otherwise, solve it for a minimizer $w^{*,k}$. Let $t = 2$.

\item [Step 2] Let $\omega:=w^{*,k}|_{2t}$.
If it satisfies \reff{cd:flat:omga}, go to Step~4; otherwise, go to Step~3.

\item [Step 3] If $t<k$, set $t:=t+1$ and go to Step~2;
otherwise, set $k:=k+1$ and go to Step~1.

\item [Step 4] Compute $c_i>0$ and $(u_i, v_i) \in K$.
Let each $a_i = c_i^{\frac 14} u_i, b_i = c_i^{\frac 14} v_i$.
Output the S-decomposition of $A$ as
\[
A = \sum_{i=1}^r   ( a_i a_i^T)  \otimes  (b_ib_i^T).
\]

\eit

\end{alg}

In Step~0, we can choose a random matrix $G$ of length $\binom{p+q+3}{3}$
and then let
\[
R = [x,y]_3^T (G^TG) [x,y]_3.
\]
Step~1 is justified by Theorem~\ref{thm:notsep}.
In Step~4, the method in Henrion and Lasserre \cite{HenLas05}
can be used to compute $c_i$ and $(u_i,v_i)$.
Indeed, Algorithm~\ref{sdpalg:A-tkmp} can be easily implemented
by the software {\tt GlotpiPoly~3} \cite{GloPol3}.

In Step~2, we need to check the rank condition in \reff{cd:flat:omga}.
In numerical computations, sometimes it may be difficult
to determine matrix ranks. This is a classical question in numerical linear algebra.
A common practice is to evaluate the rank as the number of singular values
larger than a threshold (e.g., $10^{-6}$).
We refer to the book \cite{Demmel} for how to
evaluate matrix ranks numerically.

\subsection{Convergence of the algorithm}

First, we study how to detect that $A$ is not separable.

\begin{theorem} \label{thm:notsep}
Let $A \in \mathscr{K}^{p,q}$ and $\mathbf{a} = A|_{\mc{E}}$
as in \reff{a=A|E}. Then we have:

\bit

\item [(i)] If \reff{min<R,w>} is infeasible for some $k$,
then $A$ is not separable, i.e., $A \not\in \mathscr{S}^{p,q}$.

\item [(ii)] If $A \not\in \mathscr{S}^{p,q}$,
then \reff{min<R,w>} is infeasible when $k$ is big enough.

\eit
\end{theorem}
\begin{proof}
(i) Suppose otherwise $A \in \mathscr{S}^{p,q}$.
Then there exist unit vectors $(u_i,v_i) \in K$ such that
\[
\mathbf{a} = \sum_i c_i [(u_i, v_i)]_{ \mc{E} }
\]
with all $c_i>0$. For all $k\geq 3$, the tms
\[
\widetilde{w} = \sum_i  c_i [(u_i, v_i)]_{ 2k }
\]
is feasible for \reff{min<R,w>},
which is a contradiction.

(ii) When $A$ is not separable, there exists a nonnegative bi-quadratic form
$B_1(x,y)$ such that
$
\langle B_1, A \rangle < 0,
$
by Proposition~\ref{prop:cone:posep}.
For $\eps>0$ small and $B_2 = B_1 + \eps (x^Tx)(y^Ty)$, we still have
$
\langle B_2, A \rangle < 0.
$
Note that $B_2(x,y)$ is strictly positive on $K$.
By Putinar's Positivstellensatz (cf.~\cite{Put}), there exists $k_0$ such that
\[
B_2 \in I_{2k_0}(h) + Q_{k_0}(g).
\]
Clearly, for all $\tau >0$, we have
\[
R - \tau (- B_2)  \in I_{2k_0}(h) + Q_{k_0}(g),
\]
\[
\langle \tau (-B_2), \mathbf{a} \rangle =
\tau \langle -B_2, A \rangle  \, \to \, + \infty
\]
as $\tau \to +\infty$.
This shows that $-B_2$ is an improving direction for \reff{sos:max<p,y>}.
Thus, \reff{sos:max<p,y>} is unbounded from above,
and \reff{min<R,w>} must be infeasible, for $k \geq k_0$.
\end{proof}

Second, we prove the asymptotic convergence of Algorithm~\ref{sdpalg:A-tkmp}.

\begin{theorem}  \label{thm:sep:asym}
Let $A \in \mathscr{S}^{p,q}$ and $\mathbf{a}$ be as in \reff{a=A|E}.
For a generic polynomial $R \in \Sig[x,y]_6$, we have:

\bit

\item [(i)] For all $k \geq 3$, the relaxation
\reff{min<R,w>} has an optimizer $w^{*,k}$.

\item [(ii)] For all $t$ sufficiently large, the truncated sequence
$\{w^{*,k}|_{2t}\}$ is bounded and all its accumulation points are
flat extensions of $\mathbf{a}$.

\eit
\end{theorem}
\begin{proof}
When $A \in \mathscr{S}^{p,q}$, the tms $\mathbf{a} = A|_{\mc{E}}$
admits a representing measure supported in $K$.

(i)
A generic $R \in \Sig[x,y]_6$
lies in the interior of $\Sig[x,y]_6$.
The conclusion can be implied by Proposition~5.1(i) of \cite{Nie-ATKMP}.

(ii) The set is contained in the ball $x^Tx+y^Ty  \leq 2$.
The conclusion can be implied by Theorem~5.3(i) of \cite{Nie-ATKMP}.
\end{proof}

Third, we investigate when Algorithm~\ref{sdpalg:A-tkmp}
converges within finitely many steps, i.e.,
when the stopping condition \reff{cd:flat:omga} is satisfied for some $k$.
Indeed, under some general conditions, the finite convergence occurs.
This is verified in all our numerical experiments.

Let $\mathscr{P}\big( K \big)$ be the cone of all polynomials that
are nonnegative on the set $K$ as in \reff{def:S+pq}.
Consider the optimization problem
\be  \label{max<f,a>:psd}
\max   \quad  \langle f, \mathbf{a} \rangle  \quad
s.t.  \quad  R - f \in \mathscr{P}\big( K \big),
\, f \in \mbox{span} \{ \mc{E} \}.
\ee
Denote by $\mbox{int}(\Sig[x,y]_6)$ the interior of $\Sig[x,y]_6$.

\begin{theorem} \label{thm:fin-cvg}
Let $A \in \mathscr{S}^{p,q}$ and $\mathbf{a}$ be as in \reff{a=A|E}.
Suppose $R \in \mbox{int}( \Sig[x,y]_6)$ and $f^*$ is a maximizer of \reff{max<f,a>:psd}.
Assume that $\hat{f}:=R-f^* \in I(h) + Q(g)$
and $\hat{f}$ has finitely many critical zeros on $x^Tx=y^Ty=1$.
For all $k$ sufficiently large,
if $w^{*,k}$ is a minimizer of \reff{min<R,w>},
then the condition \reff{cd:flat:omga} must be satisfied.
\end{theorem}
\begin{proof}
When $R \in \mbox{int}(\Sig[x,y]_6)$,
the feasible set of \reff{sos:max<p,y>} has an interior point.
By Proposition~5.1 of \cite{Nie-ATKMP}, the optimization problems
\reff{min<R,w>} and \reff{sos:max<p,y>} have equal optimal values.
By the assumption, there exists $k_1$ such that
\[
\hat{f} \in I_{2k_1}(h) + Q_{k_1}(g).
\]
Note that
$
I_{2k}(h) + Q_{k}(g)  \subseteq
\mathscr{P}\big( K \big)
$
for all $k$. Hence, for all $k \geq k_1$,
$f^*$ is a maximizer of \reff{sos:max<p,y>}, and
\[
\langle R, w^{*,k} \rangle =  \langle f^*, \mathbf{a} \rangle =
  \langle f^*, w^{*,k}  \rangle.
\]
Then,
\[
\langle \hat{f}, w^{*,k} \rangle = 0
\quad \forall \, k\geq k_1.
\]
Since $\hat{f} \in I_{2k_1}(h)+Q_{k_1}(g)$,
$\hat{f}$ is nonnegative on $K$.
The dual problem of \reff{max<f,a>:psd} is
\be  \label{min<R,w>:sep}
\min   \quad  \langle R, z \rangle  \quad
s.t.  \quad  z|_{\mc{E}} = \mathbf{a}, \quad
\, z \in \mathscr{R}_6(K).
\ee
(The symbol $\mathscr{R}_6(K)$ denotes the closed convex cone
of vectors in $\re^{\mbM[x,y]_6}$ that
admit representing measures supported in $K$.)
The strong duality holds between
\reff{max<f,a>:psd} and \reff{min<R,w>:sep},
because $R \in \mbox{int}( \Sig[x,y]_6)$.
Since $A \in \mathscr{S}^{p,q}$, $\mathbf{a}$ admits a representing
measure supported on $K$, so \reff{min<R,w>:sep}
must have a minimizer (say, $z^*$).
Let $\mu$ be a $K$-representing measure for $z^*$, then,
\[
0 = \langle R, z^* \rangle - \langle f^*, \mathbf{a} \rangle
= \langle \hat{f}, z^* \rangle = \int \hat{f} \mathtt{d} \mu.
\]
This implies that the minimum value of $\hat{f}$ on $K$ is zero.

Consider the polynomial optimization problem:
\be \label{pop:min:hat(f)}
\min_x \quad \hat{f}(x) \quad
s.t.  \quad h(x) = 0, \,  g(x) \geq 0.
\ee
The $k$-th order SOS relaxation for \reff{pop:min:hat(f)} is
\be \label{sosmax:hat(f)-gm}
f_{1,k} := \max  \quad \gamma \quad s.t. \quad
\hat{f}  -  \gamma \in  I_{2k}(h) + Q_k(g).
\ee
Its dual problem is
\be \label{mommin:<hat(f),w>}
\left\{\baray{rl}
f_{2,k} := \min\limits_{ w } & \langle \hat{f}, w \rangle  \\
s.t. &  \langle 1, w \rangle  = 1, M_k(w) \succeq 0, \\
&  L_{h}^{(k)}(w) = 0,  L_{g}^{(k)}(w) \succeq 0.
\earay \right.
\ee
Since $\hat{f} \in I_{2k_1}(h)+Q_{k_1}(g)$, we have
$f_{1,k} \geq 0$ for all $k \geq k_1.$
On the other hand, the minimum value of $\hat{f}$ on $K$ is $0$,
so $f_{1,k} \leq 0$ for all $k$. Hence,
\[
f_{1,k} = 0 \quad \forall \, k \geq k_1.
\]
Lasserre's hierarchy for \reff{pop:min:hat(f)} has finite convergence.
The problem \reff{sosmax:hat(f)-gm} achieves its optimal value for $k \geq k_1$,
because $\hat{f} \in I_{2k_1}(h)+ Q_{k_1}(g)$.

%%%%%%%%%%%%%%%%%%%%%%%%%%%
\iffalse

For a polynomial $f$,
denote by $f^{hom}$ the homogeneous part of $f$ with the highest degree.
For a tuple $(a_1,\ldots, a_r)$ of forms,
$\Dt(a_1,\ldots, a_r)$ denotes its discriminant.
Note that $\Dt(a_1,\ldots, a_r)=0$ if and only if
$a_1,\ldots, a_r$ have a common zero, in the projective space,
on which the Jacobian of $a_1,\ldots, a_r$ is rank deficient.
We refer to \cite[Section~3]{Nie-dis} for the definition of discriminants.

Since $\deg(R)=6$, $(R-f^*)^{hom} = R^{hom}$. Note that
\[
\Dt(z^Tz) \ne 0, \Dt(z^Tz, \mathbf{1}_p^Tx) \ne 0,
\]
\[
\Dt(z^Tz, \mathbf{1}_q^Ty) \ne 0,
\Dt(z^Tz, \mathbf{1}_p^Tx, \mathbf{1}_q^Ty) \ne 0.
\]
When $R$ is generic in $\Sig[x,y]_6$, then
\[
\Dt(R^{hom}, z^Tz) \ne 0, \Dt(R^{hom}, z^Tz, \mathbf{1}_p^Tx) \ne 0,
\]
\[
\Dt(R^{hom}, z^Tz, \mathbf{1}_q^Ty) \ne 0,
\Dt(R^{hom}, z^Tz, \mathbf{1}_p^Tx, \mathbf{1}_q^Ty) \ne 0.
\]
By Proposition~A.1 in the Appendix of \cite{Nie-ATKMP},
\reff{pop:min:hat(f)} has only finitely many critical points.

\fi
%%%%%%%%%%%%%%%%%%%%%%%%%%%%%%%%%%%%

When $(w^{*,k})_{0}=0$, then $vec(1)^T M_k(w^{*,k}) vec(1) =0$,
and $M_k(w^{*,k}) vec(1) =0$ because $M_k(w^{*,k})\succeq 0$.
(Here $vec( )$ denotes the coefficient vector.)
Moreover, we have $M_k(w^{*,k}) vec(z^\af)=0$ for all $|\af| \leq k-1$
(cf.~\cite[Lemma~5.7]{Lau}). So, for $k\geq 3$, $w^{*,k}|_4$ is flat.

When $(w^{*,k})_0>0$, there exists $\tau>0$ such that $(\tau w^{*,k})_0=1$.
Let $w^* =\tau w^{*,k}$.
Then $w^*$ is a minimizer of \reff{mommin:<hat(f),w>},
because $\langle \hat{f}, w^* \rangle = 0$ for all $k\geq k_1$.
By the assumption, $\hat{f}$ has finitely many critical zeros
on $x^Tx=y^Ty=1$, so Assumption~2.1 in \cite{Nie-ft}
for \reff{pop:min:hat(f)} is satisfied.
By Theorem~2.2 of \cite{Nie-ft},
$w^*$ has a flat truncation $w^*|_{2t}$ if $k$ is big enough,
and so is $w^{*,k}$.
\end{proof}

If a polynomial $\sig$ is nonnegative on $K$,
then we often have $\sig \in I(h) + Q(g)$,
under some general conditions (cf.~\cite{Nie-opcd}).
For instance, this is the case if the standard optimality conditions
(constraint qualification, second order sufficiency, strict complementarity)
hold. These optimality conditions are generically satisfiable (cf. \cite{Nie-opcd}).
So, the assumption $\hat{f} \in I(h) + Q(g)$ in
Theorem~\ref{thm:fin-cvg} is often satisfied.
Thus, Algorithm~\ref{sdpalg:A-tkmp} typically has finite convergence.
In all our numerical experiments, the finite convergence always occured.

\subsection{A comparision}

We would like to make a comparison between Algorithms~\ref{alg:posmap}
and \ref{sdpalg:A-tkmp}. By Proposition~\ref{prop:cone:posep},
the positive map cone $\mathscr{P}^{p,q}$
and the separable matrix cone $\mathscr{S}^{p,q}$
are dual to each other. One may expect that their memberships
can be checked in similar ways.
However, these two algorithms have slightly different
properties for checking the memberships of
$\mathscr{S}^{p,q}$ and $\mathscr{P}^{p,q}$.
For {\it every} linear map $\Phi$,
Algorithm~\ref{alg:posmap} is able to determine whether
$\Phi$ belongs to $\mathscr{S}^{p,q}$ or not,
within finitely many steps
(cf.~Theorem~\ref{thm:posmap:cvg}(i)).

In the constrast, for the cone $\mathscr{S}^{p,q}$,
we have a slightly weaker conclusion.
For a matrix $\mathscr{K}^{p,q}$, if $A \not\in \mathscr{S}^{p,q}$,
then Algorithm~\ref{sdpalg:A-tkmp} is able to verify
$A \not\in \mathscr{S}^{p,q}$, within finitely many steps
(cf.~Theorem~\ref{thm:notsep}(ii)).
However, if $A \in \mathscr{S}^{p,q}$, Algorithm~\ref{sdpalg:A-tkmp}
is able to get an S-decomposition asymptotically
(cf.~Theorem~\ref{thm:sep:asym}(ii)).

However, Algorithm~\ref{sdpalg:A-tkmp} has finite convergence,
under some additional conditions (cf.~Theorem~\ref{thm:fin-cvg}).
Interestingly, such conditions are generally satisfied
(see the comments after the proof of Theorem~\ref{thm:fin-cvg}).
In other words, Algorithm~\ref{sdpalg:A-tkmp}
is almost always able to check the membership of
$\mathscr{S}^{p,q}$, within finitely many steps.
This was confirmed in our numerical experiments.

A mathematical reason for the above difference is as follows.
To test positive maps, Algorithms~\ref{alg:posmap}
can make use of the optimality condition \reff{kkt:Bxy}.
However, to test separable matrices, there is no such
a convenient condition to use for Algorithm~\ref{sdpalg:A-tkmp}.
This is why checking separable matrices is often
harder than checking linear positive maps.

At the moment, we are not able to prove that
Algorithm~\ref{sdpalg:A-tkmp} always have finite convergence
for all matrices $A \in \mathscr{K}^{p,q}$.
However, no matrices were found such that
Algorithm~\ref{sdpalg:A-tkmp} fails to terminate
after a finite number of steps.
To the best of the authors' knowledge, this is an open question.

\section{Numerical Examples}
\label{sc:nmexp}
\setcounter{equation}{0}

In this section, we present some examples for checking positivity of linear maps and
separability of matrices. The computation is implemented
in 64-bit MATLAB R2012a, on a Lenovo
Laptop with Intel(R) Core(TM)i7-3520M CPU@2.90GHz and RAM 16.0G.
Algorithms~\ref{alg:posmap} and \ref{sdpalg:A-tkmp} can be implemented
by the software {\tt GloptiPoly~3} \cite{GloPol3}, which calls
the SDP solver {\tt SeDuMi} \cite{sedumi}.
In the computation, the rank of a matrix is numerically evaluated
as the number of its singular values that
is bigger than $10^{-6}$.
For computational results, only four decimal digits are displayed,
for cleanness of the presentation.

\subsection{Checking positivity of linear maps}

\begin{example}\label{example5.1}
(\cite[Example 5.1]{HQ13})
Consider the linear map $\Phi: \mc{S}^2 \to \mc{S}^2$ such that
\[
y^T \Phi(xx^T) y =
\bpm x_1 y_1 \\ x_1y_2 \\ x_2y_1 \\ x_2y_2 \epm^T
\left(\begin{array}{rrrr}
0.0058&-0.1894&-0.2736&0.3415 \\
-0.1894&-0.1859&-0.1585&0.0841\\
-0.2736&-0.1585&-0.0693&-0.0669\\
0.3415&0.0841&-0.0669&0.2494
\end{array}\right)
\bpm x_1 y_1 \\ x_1y_2 \\ x_2y_1 \\ x_2y_2 \epm.
\]
By solving the semidefinite relaxation (\ref{min<B,w>:mom}) with $k=3$,
we get the optimal value of \reff{minB(xy):2sph}
$b_{min} = -0.3157$, as well as a minimizer $(x^*,y^*)$
\[
\big( (0.9830,\,  -0.1835),\,  (0.4632,\,   0.8863) \big).
\]
This linear map is not positive.
The consumed computational time is around $0.8$ second
and the rank of the moment matrix is $1$.
%%%%%%%%%%%%%%%%%%%%%%%%%%%%%%%%%%%%%%%%%%%
\iffalse

clear all
B=[0.0058,-0.1894,-0.2736,0.3415; ...
    -0.1894,-0.1859,-0.1585,0.0841; ...
    -0.2736,-0.1585,-0.0693,-0.0669; ...
    0.3415,0.0841,-0.0669,0.2494];

p=2;q=2;
mpol x 2
mpol y 2
z=[x(1)*y; x(2)*y];
g=0;
for i=1:4
    for j=1:4
        g=g+B(i,j)*z(i)*z(j);
    end
end

Bx=diff(g,x);
By=diff(g,y);
K=[x'*x-1==0, y'*y-1==0,Bx-2*g*x'==0,By-2*g*y'==0, sum(x)>=0, sum(y)>=0];
P=msdp(min(g),K,3);
[status,obj,Mea]=msol(P);

mu=meas;mv=mvec(mu);M3=double(mu); M2=M(1:15,1:15);
svd(M2),svd(M3)

\fi
%%%%%%%%%%%%%%%%%%%%%%%%%%%%%%%%%%%
\end{example}

\begin{example} (\cite[\S4]{QDH09})\label{example5.3}
Consider the linear map $\Phi: \mc{S}^2 \to \mc{S}^2$ such that
\[
y^T \Phi(xx^T)y =
x_1^2(y_1^2+4y_1y_2+12y_2^2)+
x_1x_2(4 y_1^2+ 16 y_1y_2+ 2 y_2^2)+
x_2^2( 12y_1^2+ 2y_1y_2 + 2y_2^2 ).
\]
%%%%%%%%%%%%%%%%%%%%%%%%%%%
\iffalse
Consider tensor $\mathcal{B}\in\mathscr{P}^{2,2}$,
 come from an example in \cite{QDH09}, with entries
 %
$$\begin{array}{lllll} b_{1111} &= 1; b_{1112} &= 2;
 b_{1212} &= 4; b_{1122} &= 12; \\
b_{2211} &= 12; b_{1222} &= 1;
 b_{1211} &= 2; b_{2212} &= 1;\\
 b_{2222} &= 2:\end{array}$$
\fi
%%%%%%%%%%%%%%%%%%%%%%%%%%%%
By solving the semidefinite relaxation (\ref{min<B,w>:mom}) with $k=3$,
we get the optimal value of \reff{minB(xy):2sph} $b_{min} = 0.5837$ and
an optimizer
\[
\big( (0.9946,-0.1040), \quad  (0.9946, -0.1040) \big).
\]
This linear map is positive. The computational time is around $0.7$ second
and the rank of the moment matrix is $1$.
%%%%%%%%%%%%%%%%%%%%%%%%%
\iffalse

clear all
p=2;q=2;
mpol x 2; mpol y 2;
Bxy = x(1)^2*( y(1)^2 + 4*y(1)*y(2) + 12*y(2)^2 ) + ...
x(1)*x(2)*( 4*y(1)^2 + 16*y(1)*y(2) + 2*y(2)^2  ) + ...
x(2)^2* ( 12*y(1)^2 + 2*y(1)*y(2) + 2*y(2)^2 );
K=[x'*x-1==0,y'*y-1==0,diff(Bxy,x)-2*Bxy*x'==0, ...
  diff(Bxy,y)-2*Bxy*y'==0, sum(x)>=0, sum(y)>=0];
P=msdp(min(Bxy),K, 3);
[status,obj,Mea]=msol(P);
obj,
double(Mea),
mu=meas;mv=mvec(mu);M3=double(mu); M2=M(1:15,1:15);
svd(M2),svd(M3)
\fi
%%%%%%%%%%%%%%%%%%%%%%%%%%%%%%%
\end{example}

\begin{example} \label{example5.4}
(\cite[Example 4.1]{WQZ09})
Consider the linear map $\Phi: \mc{S}^3 \to \mc{S}^3$ such that
\[
y^T \Phi(xx^T) y = \sum_{1 \leq i,j,k, l \leq 3} f_{ijkl} x_iy_j x_k x_l,
\]
where the coefficients $f_{ijkl}$ satisfy the symmetric pattern
\[
f_{ijkl} =  f_{k l i j} = f_{kjil} =  f_{ilkj}
\]
and are given as
\[
 \begin{array}{llccr}
 &f_{1111}=-0.9727;~~ &f_{1112}=0.3169; ~~&f_{1113}=-0.3437;~~&f_{1121}=0.3169;\\
 &f_{1122}=0.6158;&f_{1123}=-0.0184;&f_{1133}=0.5649;&f_{1211}=-0.6332;\\
 &f_{1212}=0.7866;&f_{1213}=0.4257;&f_{1222}=0.0160;&f_{1223}=0.0085;\\
 &f_{1233}=-0.1439;&f_{1311}=0.3350;&f_{1312}=-0.9896;&f_{1313}=-0.4323;\\
 &f_{1322}=-0.6663;&f_{1323}=0.2599;&f_{1333}=0.6162;&f_{2211}=0.7387;\\
 &f_{2212}=0.6873;&f_{2213}=-0.3248;&f_{2222}=0.5160;&f_{2223}=-0.2160;\\
 &f_{2233}=-0.0037;&f_{2311}=-0.7986;&f_{2312}=-0.5988;&f_{2313}=-0.9485;\\
 &f_{2322}=0.0411;&f_{2323}=0.9857;&f_{2333}=-0.7734;&f_{3311}=0.5853;\\
 &f_{3312}=0.5921;
 &f_{3313}=0.6162;
 &f_{3322}=-0.2907;
 &f_{3323}=-0.3881;\\
 &f_{3333}=-0.8526;
\end{array}
\]
By Algorithm~\ref{alg:posmap} with $k=3$, we get the optimal value
of \reff{minB(xy):2sph} $b_{min} = -2.3197$,
and a minimizer $(x^*, y^*)$:
\[
\big( ( -0.3496,\,  -0.4003,\,    0.8471),\,
(-0.5017 ,\,   0.5383 ,\,   0.6772) \big).
\]
This linear map is not positive. The computational time is
around $3$ seconds and the rank of the moment matrix is $1$.
%%%%%%%%%%%%%%%%%%%%%%%%%%%%%%%%%%%%%%%
\iffalse
clear all
A=zeros(3,3,3,3);
A(:,:,1,1)=[-0.9727,0.3169,-0.3437;
    -0.6332,-0.7866,0.4257;
    -0.3350,-0.9896,-0.4323];
A(:,:,2,1)=[-0.6332,-0.7866,0.4257;
    0.7387,0.6873,-0.3248;
    -0.7986,-0.5988,-0.9485];
A(:,:,3,1)=[-0.3350,-0.9896,-0.4323;
    -0.7986,-0.5988,-0.9485;
    0.5853,0.5921,0.6301];
A(:,:,1,2)=[0.3169,0.6158,-0.0184;
    -0.7866,0.0160,0.0085;
    -0.9896,-0.6663,0.2559];
A(:,:,2,2)=[-0.7866,0.0160,0.0085;
    0.6873,0.5160,-0.0216;
    -0.5988,0.0411,0.9857];
A(:,:,3,2)=[-0.9896,-0.6663,0.2559;
    -0.5988,0.0411,0.9857;
    0.5921,-0.2907,-0.3881];
A(:,:,1,3)=[-0.3437,-0.0184,0.5649;
    0.4257,0.0085,-0.1439;
    -0.4323,0.2559,0.6162];
A(:,:,2,3)=[0.4257,0.0085,-0.1439;
    -0.3248,-0.0216,-0.0037;
    -0.9485,0.9857,-0.7734];
A(:,:,3,3)=[-0.4323,0.2559,0.6162;
    -0.9485,0.9857,-0.7734;
    0.6301,-0.3881,-0.8526];
p=3;q=3;
mpol x 3
mpol y 3

g=0;
for i=1:p
    for j=1:q
        for k=1:p
            for l=1:q
                g=g+A(i,j,k,l)*x(i)*y(j)*x(k)*y(l);
            end
        end
    end
end

Bx=diff(g,x); By=diff(g,y);
K=[x'*x-1==0, y'*y-1==0,Bx-2*g*x'==0,By-2*g*y'==0,sum(x)>=0, sum(y)>=0];
P=msdp(min(g),K, 3);
[status,obj,Mea]=msol(P);
double(Mea)',
obj,

mu=meas;mv=mvec(mu);M3=double(mu); M2=M(1:28,1:28);
svd(M2),svd(M3)

\fi
%%%%%%%%%%%%%%%%%%%%%%%%%%%%%%%%%%%%%%
\end{example}

\begin{example}\label{example5.6}
(\cite{LNQY09})
Consider the linear map $\Phi:\mc{S}^3 \to \mc{S}^3$ such that
\[
\baray{c}
y^T \Phi(xx^T) y =x_1^2y_1^2+x_2^2y_2^2+x_3^2y_3^2+2(x_1^2y_2^2+x_2^2y_3^2+x_3^2y_1^2) \\
-2(x_1x_2y_1y_2+x_1x_3y_1y_3+x_2x_3y_2y_3).
\earay
\]
By solving the semidefinite relaxation \reff{min<B,w>:mom} with $k=3$,
we get the optimal value of \reff{minB(xy):2sph} $b_{min}=0$
and $3$ minimizers:
\[
\big( (0, 1, 0), (1, 0, 0) \big) , \quad
\big( (0, 0, 1), (0, 1, 0) \big) , \quad
\big( (1, 0, 0), (0, 0, 1) \big).
\]
This linear map is positive.
The convex relaxation in \cite{LNQY09} is not tight
for checking positivity of this map. The computational time is
around $5$ seconds and the rank of the moment matrix is $3$.
%%%%%%%%%%%%%%%%%%%%%%%%%%%%
\iffalse

clear all
mpol x 3;
mpol y 3;
g=x(1)^2*y(1)^2+x(2)^2*y(2)^2+x(3)^2*y(3)^2+2*(x(1)^2*y(2)^2+ ...
x(2)^2*y(3)^2+x(3)^2*y(1)^2)-2*x(1)*x(2)*y(1)*y(2)-2*x(1)*x(3)*y(1)*y(3)-2*x(2)*x(3)*y(2)*y(3);
Bx=diff(g,x);
By=diff(g,y);
K=[x'*x-1==0, y'*y-1==0,Bx-2*g*x'==0,By-2*g*y'==0, sum(x)>=0, sum(y)>=0];
P=msdp(min(g),K, 3);
[status,obj,Mea]=msol(P);
solxy = double(Mea),
solxy(:,:,1)',
solxy(:,:,2)',
solxy(:,:,3)',
obj,
mu=meas;mv=mvec(mu);M3=double(mu); M2=M(1:28,1:28);
svd(M2),svd(M3)
\fi
%%%%%%%%%%%%%%%%%%%%%%%%%%%%%%%%%
\end{example}

\begin{example}
Consider the linear map $\Phi:\mc{S}^4 \to \mc{S}^4$ such that
\[
y^T \Phi(xx^T) y = \sum_{  1 \leq i \leq k \leq 4 , 1 \leq j \leq l \leq 4  }
\frac{x_iy_j x_k y_l}{i+j+k+l}.
\]
By solving the semidefinite relaxation \reff{min<B,w>:mom} with $k=3$,
we get the optimal value of \reff{minB(xy):2sph} $b_{min} = 0.0175$
and also a minimizer:
\[
\big( (-0.0565,\,   -0.1415,\,    -0.5192,\, 0.8410),\,
(-0.0565,\,    -0.1415,\,    -0.5192,\,     0.8410) \big).
\]
This linear map is positive. The computational time is around
$116$ seconds and the rank of the moment matrix is $1$.
%%%%%%%%%%%%%%%%%%%%%%%%%%%%
\iffalse

clear all
p=4; q = 4;
mpol x 4; mpol y 4;
Bxy = 0;
for i = 1 : p
   for j = 1 : q
     for k = i : p
        for l = j : q
           Bxy = Bxy + 1/(i+j+k+l)*x(i)*y(j)*x(k)*y(l);
end, end, end, end
Bx=diff(Bxy,x);
By=diff(Bxy,y);
K=[x'*x-1==0, y'*y-1==0,Bx-2*Bxy*x'==0,By-2*Bxy*y'==0, sum(x)>=0, sum(y)>=0];
P=msdp(min(Bxy),K, 3);
[status,obj,Mea]=msol(P);
solxy = double(Mea),
obj,
mu=meas;mv=mvec(mu);M3=double(mu); M2=M(1:45,1:45);
svd(M2),svd(M3)
\fi
%%%%%%%%%%%%%%%%%%%%%%%%%%%%%%%%%
\end{example}

\subsection{Numerical examples of decomposition of separable matrices}

\begin{example}
(\cite[Example~5.1]{HQ13}) \label{example5.10}
Consider the matrix in $\mathcal{K}^{2,2}$:
\[
A=\left[\begin{array}{rrrr}
0.4691& 0.1203& -0.1203& 0.4691\\
0.1203& 0.0309& -0.0309& 0.1203\\
-0.1203& -0.0309& 0.0309& -0.1203\\
0.4691& 0.1203& -0.1203& 0.4691
\end{array}\right].
\]
%%%%%%%%%%%%%%%%%%%%%%
\iffalse

clear all,
p=2; q=2;
A = [0.4691,0.1203,-0.1203,0.4691; ...
0.1203,0.0309,-0.0309,0.1203; ...
-0.1203,-0.0309,0.0309,-0.1203; ...
0.4691,0.1203,-0.1203,0.4691];
mpol x 2; mpol y 2;
Kmom = [];
for i = 1 : p,  for j = 1 : q
       for k = i : p,   for l = j : q
   Kmom = [ Kmom, mom( x(i)*y(j)*x(k)*y(l) ) == A( (i-1)*q+j, (k-1)*q+l) ];
       end,  end
end, end
Kxy = [x'*x==1, y'*y == 1, sum(x)>=0,  sum(y) >=0];
monxy = mmon([x; y],0,3);  L1 =  randn(length(monxy));
Rmat = (L1'*L1)/1e3; Rxy = monxy'*Rmat*monxy;
P = msdp( min( mom(Rxy) ), Kmom, Kxy, 3);
[sta, objval, Mea] = msol(P);

mu=meas;M=double(mmat(mu));eig(M*M')

\fi
%%%%%%%%%%%%%%%%%%%%%%
The semidefinite relaxation (\ref{min<R,w>}) is infeasible for $k=3$,
so $A$ is not separable, i.e., $A\not\in \mathscr{S}^{2,2}$.
The computational time is around $1$ second.
\end{example}

\begin{example}
Consider the matrix $A = A_1 + 2 A_2 - \half A_3$ in $\mathcal{K}^{3,3}$ where
\[
\baray{rcl}
A_1  &=& (e_1e_1^T) \otimes (e_1e_1^T) + (e_2e_2^T) \otimes (e_2e_2^T) +
 (e_3e_3^T) \otimes (e_3e_3^T), \\
A_2  &=& (e_1e_1^T) \otimes (e_2e_2^T) + (e_2e_2^T) \otimes (e_3e_3^T) +
 (e_3e_3^T) \otimes (e_1e_1^T), \\
A_3  &=& (e_1e_2^T+e_2e_1^T) \otimes  (e_1e_2^T+e_2e_1^T) +
(e_1e_3^T+e_3e_1^T) \otimes  (e_1e_3^T+e_3e_1^T) \\
&& +(e_2e_3^T+e_3e_2^T) \otimes  (e_3e_2^T+e_2e_3^T).
\earay
\]
One can check that $\langle A, (xx^T)\otimes (yy^T) \rangle$
is the polynomial in Example~\ref{example5.6}.
The semidefinite relaxation (\ref{min<R,w>}) is infeasible for $k=3$,
so $A$ is not separable, i.e., $A\not\in \mathscr{S}^{3,3}$.
The computational time is around $6$ seconds.
%%%%%%%%%%%%%%%%%%%%%%
\iffalse

clear all,
p=3; q=3;
e1 = [1 0 0]'; e2 = [0 1 0]'; e3 = [0 0 1]';
A1 = kron( e1*e1', e1*e1') +  kron( e2*e2', e2*e2') +  kron( e3*e3', e3*e3');
A2 = kron( e1*e1', e2*e2') +  kron( e2*e2', e3*e3') +  kron( e3*e3', e1*e1');
A3 = kron( e1*e2'+e2*e1', e1*e2'+e2*e1' ) + kron( e1*e3'+e3*e1', e1*e3'+e3*e1' ) + ...
 kron( e2*e3'+e3*e2', e2*e3'+e3*e2' );
A = A1 + 2*A2 - (1/2)*A3;
mpol x 3; mpol y 3;  Kmom = [];
for i = 1 : p,  for j = 1 : q
       for k = i : p,   for l = j : q
   Kmom = [ Kmom, mom( x(i)*y(j)*x(k)*y(l) ) == A( (i-1)*q+j, (k-1)*q+l) ];
       end,  end
end, end
Kxy = [x'*x==1, y'*y == 1, sum(x)>=0,  sum(y) >=0];
monxy = mmon([x; y],0,3);  L1 =  triu( rand(length(monxy)) );
Rmat = (L1'*L1); Rxy = monxy'*Rmat*monxy;
P = msdp( min( mom(Rxy) ), Kmom, Kxy, 3);
[sta, objval, Mea] = msol(P);

\fi
%%%%%%%%%%%%%%%%%%%%%%%%%%%%%%%%%%%%%%%%%%%
\end{example}

\begin{example}
Consider the matrix $A \in \mathcal{K}^{4,4}$ such that
\[
A_{(i-1)q+j, (k-1)q+l}  =  i+j+k+l
\]
for all $1 \leq i,j,k,l \leq 4$. The semidefinite relaxation (\ref{min<R,w>})
is infeasible for $k=3$,
so $A$ is not separable, i.e., $A\not\in \mathscr{S}^{4,4}$.
The computational time is around $56$ seconds.
%%%%%%%%%%%%%%%%%%%%%%
\iffalse

clear all,
p=4; q=4;
A = zeros(p*q);
for i = 1 : p,  for j = 1 : q
       for k = 1 : p,   for l = 1 : q
               A( (i-1)*q+j, (k-1)*q+l ) = i+j+k+l;
       end,  end
end, end
mpol x 4; mpol y 4;  Kmom = [];
for i = 1 : p,  for j = 1 : q
       for k = i : p,   for l = j : q
   Kmom = [ Kmom, mom( x(i)*y(j)*x(k)*y(l) ) == i+j+k+l ];
       end,  end
end, end
Kxy = [x'*x==1, y'*y == 1, sum(x)>=0,  sum(y) >=0];
monxy = mmon([x; y],0,3);  L1 =  round( 10*randn(length(monxy)) )/10;
Rmat = (L1'*L1); Rxy = monxy'*Rmat*monxy;
P = msdp( min( mom(Rxy) ), Kmom, Kxy, 3);
tic,
[sta, objval, Mea] = msol(P);
comptime = toc,

\fi
%%%%%%%%%%%%%%%%%%%%%%%%%%%%%%%%%%%%%%%%%%%
\end{example}

\begin{example}
Consider the following matrix $A$ in the space $\mathcal{K}^{2,3}$:
\[
A = \bbm 2  & 1 \\  1 & 3 \ebm \otimes
\bbm 3 & -1 & -1 \\ -1 & 3 & -1 \\ -1 & -1 & 3 \ebm  +
\bbm 1 & -1 \\  -1  & 2 \ebm \otimes
\bbm 4 & 2 & -1 \\  2 & 4 & 2 \\ -1 & 2 & 4 \ebm.
\]
It is separable. By Algorithm~\ref{sdpalg:A-tkmp},
we got an $S$-decomposition
$A = \sum_{i=1}^{7} (a_i a_i^T) \otimes (b_i b_i^T)$, where
$(a_i, b_i)$ are listed column by column as follows:
{\small
\begin{verbatim}
    1.2078    1.0746   -1.0379    1.2993    1.1104   -1.3520    0.5378
    1.3514    0.9620    1.6754   -1.2993    1.6509    1.4560    1.6012
    0.1118    0.5916    0.9481    1.6192    1.6265   -0.6348    0.7998
    1.2220    0.7327    1.0439    0.5969   -0.9708    1.4657    1.0804
   -1.3338   -1.0924   -1.3767    0.6311   -0.6086    1.1818    1.0229
\end{verbatim} \noindent}The
computational time is around $3$ seconds,
and the rank of the moment matrix is $7$.

\end{example}

\begin{example}
Consider the following matrix $A$ in the space $\mathcal{K}^{3,3}$:
\[
A = I_3 \otimes I_3 + (e_1e_1^T) \otimes (e_2e_2^T) +
(e_2e_2^T) \otimes (e_3e_3^T) +  (e_3e_3^T) \otimes (e_1e_1^T).
\]
It is separable. By Algorithm~\ref{sdpalg:A-tkmp}, we got an $S$-decomposition
$A = \sum_{i=1}^{15} (a_ia_i^T) \otimes (b_i b_i^T)$, where
$(a_i,b_i)$ are listed column by column as follows:
{\scriptsize
\begin{verbatim}
    0.3332    0.2690    1.0893    0.6254   -0.7835    0.4637    0.2487    0.7692
    0.3514   -0.8466   -0.2597    0.5751    0.3076    0.6064   -0.6125   -1.2164
   -0.6846    0.5776   -0.8295   -1.2005    0.4759    0.1940    0.3639    0.4472
    0.5247    0.3001    0.2107    0.6835    0.0702    0.1801   -0.4733    0.1722
   -0.1736    0.4736    0.6881   -0.0122    0.5405    0.7896    0.1044    0.3527
   -0.3512   -0.7737   -0.0089    0.3060   -0.6107   -0.9697    0.3689    0.5320

    0.4306    0.5356    0.6862   -0.1301    0.7654    0.8684    0.0262
    0.1275   -0.5190   -0.1012    0.4497    0.1233    0.6715    0.8565
    0.0838    0.4547    0.3925    0.6624    0.7719   -0.1016    0.6082
   -0.2720   -0.5632   -0.7107   -1.1051    0.7074   -0.1978    0.5716
    0.5406    0.5401    0.8189    0.4488    0.5923    0.6990    0.0517
   -0.2686    0.0231   -0.1082    0.6563    0.0799    0.5970    0.7607
\end{verbatim} \noindent}The
computational time is around $7$ seconds,
and the rank of the moment matrix is $15$.

\end{example}

In the following, we consider some randomly generated
separable matrices.

\begin{example}
Consider the following matrix $A$ in the space $\mathcal{K}^{3,4}$:
\[
A = \sum_{i=1}^{5} (u_i u_i^T) \otimes (v_i v_i^T),
\]
where $(u_1,v_1), \ldots, (u_5,v_5)$ are given column by column as
{\scriptsize
\begin{verbatim}
    1.2058    0.9072    1.7107   -0.5053    0.4015
   -0.7758   -0.4990    1.2737   -0.7534    0.7230
   -0.8226   -1.6610    0.0580    1.6702   -1.6482
    0.8679   -0.7584   -2.0588    0.0188   -1.1817
    0.4465    0.6656   -2.5623   -0.0524   -1.0712
    0.4539   -0.1715    0.3518    0.6462    0.6615
    1.1036    0.0342   -1.1263    0.7462    0.5727
\end{verbatim}
\noindent}Clearly,
$A$ is separable. By Algorithm~\ref{sdpalg:A-tkmp}, we got an $S$-decomposition
$A = \sum_{i=1}^{5} (a_i a_i^T) \otimes (b_i b_i^T)$, where
$(a_i,b_i)$ are displayed column by column as follows:
{\scriptsize
\begin{verbatim}
   -0.3476   -0.6388   -1.1734   -0.3886    2.0908
   -0.5183    0.3514    0.7547   -0.6988    1.5567
    1.1491    1.1697    0.8008    1.5939    0.0709
    0.0274    1.0770    0.8920    1.2222    1.6845
   -0.0761   -0.9452    0.4591    1.1077    2.0965
    0.9396    0.2435    0.4662   -0.6839   -0.2878
    1.0850   -0.0486    1.1338   -0.5919    0.9215
\end{verbatim}
\noindent}The computational time
is around $53$ seconds,
and the rank of the moment matrix is $5$. The computed $S$-decomposition
is same as the input one, up to a permutation and scaling of $a_i, b_i$.
That is, there exist real numbers $\tau_{i,j}$,
with $i=1,\ldots, 5$ and $j=1,2$ such that
each $| \tau_{i,1} \tau_{i,2} | = 1$ and
\[
u_i = \tau_{i,1} a_{\sig_i}, \quad v_i = \tau_{i,2} b_{\sig_i}.
\]
In the above, the permutation vector $\sig = (3, 2, 5, 1, 4)$.
\end{example}

\begin{example}
Consider the matrix in the space $\mathcal{K}^{4,4}$:
\[
A = \sum_{i=1}^{6} (u_i u_i^T) \otimes (v_i v_i^T),
\]
where $(u_1,v_1), \ldots, (u_6,v_6)$ are given as as {\scriptsize
\begin{verbatim}
   -1.6002    1.5428   -1.3328   -0.5149    0.1403    0.6616
    1.3773    1.0162   -0.4031    0.8267   -0.4983   -0.2561
   -1.8003   -2.2759   -0.4736    1.1673    1.9594    1.0980
    1.1086    0.9578   -1.5677    0.9943    0.6987   -0.6716
   -0.2947    0.8312   -0.3316   -0.3028   -1.7391   -1.4154
   -0.6738    1.0141    0.0581    0.2061   -0.3607    1.4899
   -0.3373   -0.3853   -1.8798   -1.1994   -0.5071    0.2920
    0.6769    1.1913   -0.9375   -0.9701   -0.2439   -0.0425
\end{verbatim}
\noindent}Clearly, $A$ is separable. By Algorithm~\ref{sdpalg:A-tkmp},
we got an $S$-decomposition
$A = \sum_{i=1}^{6} (a_i a_i^T) \otimes (b_i b_i^T)$,
where $(a_i,b_i)$ are displayed column by column as follows:
{\scriptsize
\begin{verbatim}
    0.9455    0.7853    1.1724    0.1316   -0.4819    1.3463
   -0.8138   -0.3040    0.7722   -0.4675    0.7737    0.4072
    1.0637    1.3033   -1.7295    1.8381    1.0925    0.4784
   -0.6550   -0.7972    0.7278    0.6555    0.9306    1.5836
    0.4988   -1.1925    1.0938    1.8538    0.3235    0.3283
    1.1403    1.2552    1.3345    0.3845   -0.2202   -0.0575
    0.5709    0.2460   -0.5070    0.5406    1.2815    1.8609
   -1.1456   -0.0358    1.5677    0.2600    1.0365    0.9281
\end{verbatim}
\noindent}The computational time is around
$110$ seconds, and the rank of the moment matrix is $6$.
The computed $S$-decomposition
is same as the input one, up to a permutation and scaling
of $a_i, b_i$.
That is, there exist real numbers $\tau_{i,j}$,
with $i=1,\ldots, 6$ and $j=1,2$ such that
each $| \tau_{i,1} \tau_{i,2} | = 1$ and
\[
u_i = \tau_{i,1} a_{\sig_i}, \quad v_i = \tau_{i,2} b_{\sig_i}.
\]
In the above, the permutation vector $\sig = (1, 3, 6, 5, 4, 2)$.
\end{example}

\subsection{Remark}

We would like to discuss the relationship of this paper
to an earlier work on bi-quadratic optimization.
Ling et al. \cite{LNQY09} proposed some convex relaxations
for bi-quadratic optimization, and proved their approximation bounds.
The relaxations in \cite{LNQY09} might not be tight (cf.~Example~\ref{example5.6}),
but provided worst case error bounds.
In contrast, the hierarchy of semidefinite relaxations constructed in this paper
is always tight for checking positive maps,
as well as for solving bi-quadratic optimization.
This is proved in Theorem~\ref{thm:posmap:cvg}.
Moreover, this paper also discusses how to check separability of matrices
and how to compute S-decompositions,
which are not the main subjects of the work \cite{LNQY09}.

%%%%%%%%%%%%%%%%%%%%%%%%%%%%%%
\iffalse

To test positivity,
numerical examples show that Algorithm~3.1 can make the decision in finitely many steps,
for {\bf all} linear maps (or bi-quadratic forms). Compared to this, some relaxation bounds (not exact solution methods) are presented in \cite{LNQY09}, which are not tight see Example 5.4. On the other hand,
 Algorithm~4.2 not only shows that a given matrix is separable,
but also compute a decomposition.  Some existing work on checking entanglement can only show that a matrix is {\it not} separable,
if this is the case. However, if a matrix is separable,
then such work cannot make the decision that
the matrix is separable. Moreover, such work cannot
compute a decomposition.

\fi
%%%%%%%%%%%%%%%%%%%%%%%%%%%%%%%%%%%%%%%%%

\bigskip
\noindent
{\bf Acknowledgement}
The authors would like to thank the associate editor
and two anonymous referees for the useful comments on improving the paper.
Jiawang Nie was partially supported by the NSF grants
DMS-0844775 and DMS-1417985.  Xinzhen Zhang was partially supported by the
National Natural Science Foundation of China (Grant No. 11471242 and 11101303)
and China Scholarship Council.


\begin{thebibliography}{100}


\bibitem{A83} M. Aron. On the role of strong ellipticity condition in nonlinear elasticity,
{\it International Journal of Engineering Science},
21(1983) pp. 1359--1367.



\bibitem{BB02}
T. Boehlke and A. Bertram.
On the ellipticity of finite isotropic linear elastic laws,
Univ., Fak. fur Maschinenbau, 2002.

%
%\bibitem{EPRRP} A.~Einstein, B.~Podolsky and N~Rosen.
%Can quantum-mechanical description of physical reality be considered complete?
%{\it Physics Review}, 47(1935) pp. 777--780.
%

%
%\bibitem{HHMH} R. Horodecki, P. Horodecki, M. Horodecki and K. Horodecki.
%{\it Quantum entanglement.}, Rev. Mod. Phys., 81(2009) pp. 865--942.
%

%
%\bibitem{NC} M.A.~Nielsen and I.L.~Chuang,
%{\it Quantum computation and quantum information},
%Cambridge University Press, 2000.
%


%
%\bibitem{AlgRAG}
%S.~Basu, R.~Pollack, and M.-F. Roy.
%{\it Algorithms in real algebraic geometry.} Series:
%Algorithms and Computation in Mathematics, Vol~10.  Springer-Verlag, 2003.
%

%
%\bibitem{Brks}
%D.~Bertsekas.
%{\it Nonlinear programming}, second edition.
%Athena Scientific, 1995.
%

%
%\bibitem{BCR}
%J. Bochnak, M. Coste and M-F. Roy.
%{\it Real algebraic geometry}, Springer, 1998.
%

\bibitem{C91} Y. Chen. On strong ellipticity and the Legendre-Hadamard condition.
{\it Archive for Rational Mechanics and Analysis}, 113(1991) pp. 165-175.
%
%\bibitem{Conway}
%Conway, John B.
%A course in Functional Analysis,
%Springer-Verlag, 1990, Second Edition.
%


%
%\bibitem{CLO97}
%D.~Cox, J.~Little and D.~O'Shea.
%{\it Ideals, varieties, and algorithms.
%An introduction to computational algebraic geometry and commutative algebra.}
%Third edition. Undergraduate Texts in Mathematics. Springer, New York, 1997.
%

%
%\bibitem{CLO98}
%D.~Cox, J.~Little and D.~O'Shea.
%{\it Using algebraic geometry}. Graduate Texts in Mathematics,
%185. Springer-Verlag, New York, 1998.
%
\bibitem{C75} M. Choi. Positive semidefinite biquadratic forms.
{\it Linear Algebra and its Applications},
12(1975) pp. 95-100.


\bibitem{CuFi05}
R. Curto and L. Fialkow.
Truncated K-moment problems in several variables.
{\it Journal of Operator Theory},  54(2005) pp. 189-226.




\bibitem{DLMO07}
G.~Dahl, J.~M.~Leinaas, J.~Myrheim, and E.~Ovrum.
 A tensor product matrix approximation problem in quantum physics.
{\it Linear Algebra and its Applications},
420(2007) pp. 711-725.



\bibitem{dKlLau11}
E.~de Klerk and M.~Laurent.
On the Lasserre hierarchy of semidefinite programming relaxations
of convex polynomial optimization problems,
{\it SIAM Journal on Optimization}, 21 (2011), pp. 824-832.


\bibitem{Demmel}
J.~Demmel.
{\em Applied Numerical Linear Algebra},
Society for Industrial and Applied Mathematics, 1997.



\bibitem{DPS04} A.C.~Doherty, P.A.~Parrilo, and F.M.~Spedalieri. Complete family of separability criteria.
{\it Physical Review A}, 69(2004) 022308.

%
%\bibitem{H97} P.~Horodecki. Separability criterion and inseparable mixed states
%with positive partial transposition. {\it Physics Review A}, 232(1997) pp. 333-339.
%

%
%\bibitem{DeSLim08}
%V. De Silva and L.-H. Lim. Tensor rank and the ill-posedness of the
%best low-rank approximation problem. {\em SIAM Journal on Matrix
%Analysis and Applications}, 30, no. 3, 1084-1127, 2008.
%


%
%
%\bibitem{GKZ}
%I. Gel’fand, M. Kapranov, and A. Zelevinsky.
%{\it Discriminants, resultants, and multidimensional determinants.}
%Mathematics: Theory \& Applications, Birkh\"{a}user, 1994.
%
%


\bibitem{Gurv} L.~Gurvits.
Classical deterministic complexity of Edmonds' Problem and quantum entanglement
{\it Proceedings of the thirty-fifth annual ACM symposium on Theory of computing},
pp. 10-19, ACM, New York, NY, USA, 2003.



\bibitem{GB02} L.~Gurvits and H.~Barnum. Largest separable balls around the maximally mixed bipartite quantum state,
     {\it Physical Review A}, 66(2002) 062311.


\bibitem{HQ13} D.R. Han and L.Q. Qi. A successive approximation method for quantum separability,
{\it Frontiers of Mathematics in China}, 8(2013) pp. 1275-1293.

\bibitem{GloPol3}
D.~Henrion, J.B.~Lasserre, and J.~Loefberg.
{\it GloptiPoly 3: moments, optimization and semidefinite programming}.
{\it Optimization Methods and Software},
24(2009) pp. 761-779.


%




\bibitem{HenLas05}
D.~Henrion and J.B.~Lasserre.
Detecting global optimality and extracting solutions in GloptiPoly.
{\it Positive polynomials in control},  293–310,
Lecture Notes in Control and Inform. Sci., 312,
Springer, Berlin, 2005.





\bibitem{KTT15} K.~Kellner, T.~Theobald, and C.~Trabandt.
A semidefinite hierarchy for containment of spectrahedra. {\it SIAM Journal on
Optimization},  25(2015) pp. 1013-1033.

\bibitem{Las01}
J.B.~Lasserre. Global optimization with polynomials and the problem of moments.
{\it SIAM Journal on Optimization}, 11(2001) pp. 796-817.


\bibitem{Las09}
J.B.~Lasserre. Convexity in semi-algebraic geometry
and polynomial optimization.
{\it SIAM Journal on Optimization}, 19(2009),  pp. 1995-2014.


\bibitem{LasBok}
J.B.~Lasserre.
{\it Moments, positive polynomials and their applications},
Imperial College Press, 2009.


\bibitem{LTY15}
J.B.~Lasserre, K.C.~Toh and S.~Yang.
A bounded degree SOS hierarchy for polynomial optimization.
{\it Euro. J. Comput. Optim.}, to appear.




\bibitem{Lau}
M. Laurent.
Sums of squares, moment matrices and optimization over polynomials.
{\it Emerging Applications of Algebraic Geometry, Vol. 149 of
IMA Volumes in Mathematics and its Applications}
(Eds. M. Putinar and S. Sullivant), Springer, pages 157-270, 2009.


\bibitem{LNQY09} C.~Ling, J.~Nie, L.~Qi, and Y.~Ye.
Biquadratic optimization over unit spheres and semidefinite programming relaxations.
 {\it SIAM Journal on Optimization}, 20(2009) pp. 1286-1310.





\bibitem{Nie-jac}
J.~Nie.
An exact Jacobian SDP relaxation for polynomial optimization.
{\it Mathematical Programming}, 137(2013) pp. 225-255..


\bibitem{Nie-ft}
J.~Nie.
Certifying convergence of Lasserre's hierarchy via flat truncation.
{\it Mathematical Programming},
142(2013) pp. 485-510.


\bibitem{Nie-opcd}
J.~Nie.
Optimality conditions and finite convergence of Lasserre's hierarchy.
{\it Mathematical Programming},
146(2014) pp. 97-121.



\bibitem{Nie-ATKMP}
J.~Nie.  The $\mc{A}$-truncated $K$-moment problem.
{\it Foundations of Computational Mathematics},
14(2014) pp. 1243-1276.



\bibitem{Nie-poprv}
J.~Nie.
Polynomial optimization with real varieties.
{\it SIAM Journal on Optimization},
23(2013) pp. 1634-1646.


\bibitem{linmomopt}
J.~Nie.
Linear optimization with cones of moments and nonnegative polynomials.
{\it Mathematical Programming}, 153(2015) pp. 247-274.

%
%\bibitem{ParMP}
%P.~Parrilo.
%Semidefinite programming relaxations for semialgebraic problems.
%{\it Math. Prog.},\,  Ser. B, Vol.~96, No.2, pp. 293-320, 2003.
%

%
%\bibitem{PR03}
%Geometry of entanglement witnesses and local detection of entanglement,
%{\it Phys. Rev. A}, 67, 012327, 2003.
%

%
%\bibitem{PS03}
%P. A.~Parrilo and B.~Sturmfels.
%Minimizing polynomial functions. In S. Basu and L. Gonzalez-Vega, editors,
%{\it Algorithmic and Quantitative Aspects of Real Algebraic
%Geometry in Mathematics and Computer Science, volume
%60 of DIMACS Series in Discrete Mathematics and
%Computer Science}, pages 83-99. AMS, 2003.
%
%


\bibitem{Put}
M. Putinar.
Positive polynomials on compact semi-algebraic sets,
{\it Indiana University Mathematics Journal}, 42(1993) pp. 969-984.

    \bibitem{QDH09} L.Q. Qi,  H.H. Dai and D.R. Han. Conditions for strong ellipticity and M-eigenvalues, {\it Frontiers of Mathematics in China}, 4(2009) pp. 349-364.
%
%\bibitem{Qi03}
%L.~Qi and K.~L. Teo.
%Multivariate polynomial minimization and its application in signal processing.
%{\it  Journal of Global Optimization},  26 (2003),  pp.~419--433.
%
%
%\bibitem{Qi10}
%L.~Qi,  G.~Yu,  and E.~X. Wu.
%Higher order positive semidefinite diffusion tensor imaging.
%{\it SIAM Journal on Imaging Sciences},  3 (2010),
%pp.~416--433.
%
%
%\bibitem{Rez00}
%B. Reznick. Some concrete aspects of Hilbert's $17^{th}$ problem.
%In {\it Contemp. Math.},  Vol.~ 253, pp. 251-272. American
%Mathematical Society, 2000.
%

%
%\bibitem{SED08}
%M.~Safey El Din.
%Computing the global optimum of a multivariate polynomial
%over the reals. {\it Proceedings of ISSAC}, 71-78, 2008.
%


%
%\bibitem{Sch05}
%C.~Scheiderer.
%Non-existence of degree bounds for weighted sums of squares representations.
%{\it Journal of Complexity}, \, 21, 823-844 (2005).
%
%\bibitem{Sch05JoA}
%C.~Scheiderer.	
%Distinguished representations of non-negative polynomials.
%{\it Journal of Algebra}, \,  289, 558-573 (2005).
%%
%
%\bibitem{Sch99}
%C.~Scheiderer.
%Sums of squares of  regular functions on real algebraic varieties.
%{\it Trans. Am. Math. Soc.}, \,  352, 1039-1069 (1999).
%

%
%\bibitem{Schw06}
%M.~Schweighofer.
%Global optimization of polynomials using gradient tentacles and sums of squares.
%{\it SIAM Journal on Optimization},
%Vol.~17, No. 3, pp. 920--942, 2006.
%


%
%\bibitem{Sch09}
%C.~Scheiderer.
%Positivity and sums of squares: A guide to recent results.
%{\it Emerging Applications of Algebraic Geometry (M. Putinar, S. Sullivant, eds.)},
%IMA Volumes Math. Appl. 149, Springer, 2009, pp. 271--324.
%


\bibitem{sedumi}
J.F.~Sturm.
SeDuMi 1.02: a MATLAB toolbox for optimization over symmetric cones.
{\it Optimization Methods and Software}, 11\&12 (1999), 625-653.
\url{http://sedumi.ie.lehigh.edu}


\bibitem{WQZ09} Y.J.~Wang, L.Q.~Qi, and X.Z.~Zhang.
A practical method for computing the largest M-eigenvalue
 of a fourth-order partially symmetric tensor.
 {\it Numerical Linear Algebra with Applications}, 16(2009) pp. 589-601.

%
%
%\bibitem{Stu02}
%B. Sturmfels. {\it Solving systems of polynomial equations}.
%CBMS Regional Conference Series
%in Mathematics, 97. American Mathematical Society, Providence, RI, 2002.
%
%%
%\bibitem{VuiSon}
%H.~H.~Vui and P.~T.~ S\'{o}n.
%Global optimization of polynomials using the
%truncated tangency variety and sums of squares.
%{\it SIAM Journal on Optimization}, Vol.~19, No.~2, pp. 941--951, 2008.
%


\end{thebibliography}
\end{document}